\IfFileExists{./prepreamble-amspreprint.sty}{\RequirePackage[packages,theorems,changes]{prepreamble-amspreprint}}{}

\makeatletter
\IfFileExists{./scoop-latex/scoop-style.sty}{\providecommand*{\input@path}{}\edef\input@path{{./scoop-latex/}\input@path}}{}
\makeatother
\PassOptionsToPackage{style = authoryear-comp, backend = biber}{biblatex}
\PassOptionsToPackage{commentmarkup = footnote}{changes}
\PassOptionsToPackage{inline}{enumitem}
\RequirePackage{algorithm}
\RequirePackage{algpseudocode}    

\documentclass[english]{amsart}

\RequirePackage{biblatex}
\RequirePackage{ltxcmds}
\IfFileExists{./preamble-amspreprint.sty}{\RequirePackage[packages,theorems,changes]{preamble-amspreprint}}{}

\usepackage{manuscript}

\IfFileExists{./postpreamble-amspreprint.sty}{\RequirePackage[packages,theorems,changes]{postpreamble-amspreprint}}{}

\makeatletter
\@ifpackageloaded{changes}{
\definechangesauthor[name = {Ronny Bergmann}, color = {TolVibrantMagenta}]{RB}
\definechangesauthor[name = {Hajg Jasa}, color = {TolVibrantOrange}]{HJ}
\definechangesauthor[name = {Paula John}, color = {TolBrightPurple}]{PJ}
\definechangesauthor[name = {Max Pfeffer}, color = {TolBrightBlue}]{MP}
}{}
\makeatother        

\makeatletter
\@ifpackageloaded{hyperref}{%
	\hypersetup{
		pdftitle = {The Intrinsic Riemannian Proximal Gradient Method for Convex Optimization},
		pdfauthor = {Ronny Bergmann, Hajg Jasa, Paula John, Max Pfeffer},
		pdfkeywords = {splitting methods, proximal methods, Riemannian manifolds, convex functions},
		pdfcreator = {Created using the Scoop Template Engine version 1.6.0.}
	}
}{
	\pdfinfo{
		/Title (The Intrinsic Riemannian Proximal Gradient Method for Convex Optimization)
		/Author (Ronny Bergmann, Hajg Jasa, Paula John, Max Pfeffer)
		/Subject ()
		/Keywords (splitting methods, proximal methods, Riemannian manifolds, convex functions)
		/Creator (Created using the Scoop Template Engine version 1.6.0.)
	}
}
\makeatother

\title[The Convex Riemannian Proximal Gradient Method]{The Intrinsic Riemannian Proximal Gradient Method for Convex Optimization}

\author[R. Bergmann]{Ronny Bergmann\orcidlink{0000-0001-8342-7218}}
\address[R. Bergmann]{Norwegian University of Science and Technology, Department of Mathematical Sciences, NO-7041 Trondheim, Norway}
\email{\detokenize{ronny.bergmann@ntnu.no}}
\urladdr{https://www.ntnu.edu/employees/ronny.bergmann}

\author[H. Jasa]{Hajg Jasa\orcidlink{0009-0002-7917-0530}}
\address[H. Jasa]{Norwegian University of Science and Technology, Department of Mathematical Sciences, NO-7041 Trondheim, Norway}
\email{\detokenize{hajg.jasa@ntnu.no}}
\urladdr{https://hajg-ijk.github.io/}

\author[P. John]{Paula John\orcidlink{0009-0009-2500-5485}}
\address[P. John]{RWTH Aachen University, Institute for Geometry and Applied Mathematics, Im Süsterfeld 2, D-52072 Aachen, Germany}
\email{\detokenize{john@igpm.rwth-aachen.de}}
\urladdr{https://www.igpm.rwth-aachen.de/team/john}

\author[M. Pfeffer]{Max Pfeffer\orcidlink{0000-0002-7739-4031}}
\address[M. Pfeffer]{Universität Potsdam, Faculty of Health Sciences Brandenburg, Am Mühlenberg 9, Building 62 (H-Lab), 14476 Potsdam – Golm, Germany}
\email{\detokenize{max.pfeffer@uni-potsdam.de}}
\urladdr{https://www.maxpfeffer.com/}

\date{\today}

\dedicatory{}

\begin{document}

% Insert the abstract.
\begin{abstract}
We consider a class of (possibly strongly) geodesically convex optimization problems on manifolds with bounded sectional curvature, where the objective function splits into the sum of a smooth and a possibly nonsmooth function.
We introduce an intrinsic convex Riemannian proximal gradient (CRPG) method that employs the manifold proximal map for the nonsmooth step, without operating in the embedding or tangent space.
A sublinear convergence rate for convex problems and a linear convergence rate for strongly convex problems are established, as well as an improved rate of convergence to $\varepsilon$-stationary points compared to the nonconvex case.
Furthermore, we derive fundamental proximal gradient inequalities that generalize the Euclidean case.
Our numerical experiments on hyperbolic spaces, manifolds of symmetric positive definite matrices, and spheres demonstrate the performance of our method, using efficient or closed-form intrinsic proximal maps.

\end{abstract}

% Insert the keywords.

% Insert the Mathematics Subject Classification.
\makeatletter
\ltx@ifpackageloaded{hyperref}{%
\subjclass[2010]{\href{https://mathscinet.ams.org/msc/msc2020.html?t=90C25}{90C25}, \href{https://mathscinet.ams.org/msc/msc2020.html?t=49Q99}{49Q99}, \href{https://mathscinet.ams.org/msc/msc2020.html?t=49M30}{49M30}, \href{https://mathscinet.ams.org/msc/msc2020.html?t=65K10}{65K10}}
}{%
}
\makeatother

% Typeset the opening page.
\maketitle

% Insert the document body.
\section{Introduction}%
\label{section:introduction}

A prominent approach to nonsmooth optimization are splitting methods, where the objective function
can be split, for example
\begin{equation}
    \label{eq:splitting}
    \argmin_{p \in \cM}
    f(p)
    ,
    \qquad
    \text{where}
    \quad
    f(p)
    =
    g(p)
    +
    h(p)
    ,
\end{equation}
where $\cM$ is a Riemannian manifold,
and both $g$ and $h$ have certain properties,
for example both are convex, lower-semicontinuous, and proper.

These splitting methods have been studied extensively in the Euclidean case,
and several methods have been generalized to the Riemannian setting.
Starting with an even more general splitting into $n$ functions instead of two,
the cyclic proximal point algorithm (CPPA)~\cite{Bacak:2014:1} is
one of the first such splitting methods on manifolds.
Another approach was introduced in~\cite{BergmannPerschSteidl:2016:1},
where the Douglas-Rachford algorithm (DRA)~\cite{DouglasRachford:1956:1} was generalized to Hadamard manifolds.
An algorithm that is known to be equivalent to DRA in the Euclidean case is the Chambolle-Pock algorithm~\cite{ChambollePock:2011:1},
which is very popular especially in signal and image processing.
This was generalized to Riemannian manifolds in~\cite{BergmannHerzogSilvaLouzeiroTenbrinckVidalNunez:2021:1},
which further led to investigations on duality on manifolds~\cite{SilvaLouzeiroBergmannHerzog:2022:1,SchielaHerzogBergmann:2024}.

Another popular method that is based on splitting is the proximal gradient method.
In the Euclidean case,~\cite[Chapter 10]{Beck:2017:1} provides a thorough overview.
Here, $g$ is assumed to be ($\lipgrad$-)smooth and $h$ to be convex and potentially nonsmooth.
In this case, the Euclidean proximal gradient method can be shown to have a convergence rate of $\cO(1/k)$ whenever $g$ is convex.

To the best of our knowledge, the authors in~\cite{ChenMaMan-ChoSoZhan:2020:1} first introduced a partially convex
proximal gradient version for the Stiefel manifold, phrasing the nonsmooth part as a constrained problem in the embedding, and requiring the nonsmooth function to be the restriction of a convex function in the embedding.
This was generalized to general Riemannian manifolds in~\cite{HuangWei:2021:1},
where the authors phrase the nonsmooth subproblem as a problem in the tangent space at the current iterate, and do not consider convexity assumptions in their general analysis.
The linear convergence of this algorithm under strong retraction convexity has been studied in~\cite{ChoiChunJungYun:2024:1}.

All these works are based on a “generalized notion of the proximal map”~\cite[(3.1)]{HuangWei:2021:1}, and solving the corresponding subproblem either in the embedding or the tangent space of the current iterate.
This general notion phrases the proximal map as a retraction-based subproblem in a tangent space, which is then solved with Euclidean optimization tools, sometimes rather advanced ones such as semismooth Newton methods.
It is also often restricted to specific manifolds, usually embedded manifolds or quotients thereof.
The convexity considerations are either retraction-based or coming from the embedding, and they usually require some kind of Lipschitz property.
Most of the convergence results in these works assume that $g$ or $h$ fulfil some convexity property, either along geodesics or along curves generated by retractions when using a retraction-based generalization of the proximal subproblem.
In the work of~\cite{HuangWei:2021:1}, a “retraction smoothness” of $g$ is required, but $h$ can be nonsmooth and nonconvex for their convergence investigations, as long as $f$ fulfils a Kurdyka-Łojasiewicz (KL) property on manifolds.
However, their algorithm employs the embedding in a way that $h$ is required to have certain properties thereon in order to solve their approximation to the proximal map.

An intrinsic proximal map that allows for optimization on manifolds which are not necessarily submanifolds of a Euclidean space has been introduced first in~\cite{FengHuangSongYingZeng:2021}.
Here, the authors restrict their investigation to Hadamard manifolds and they assume the objective to have the KL property.
In a related work,~\cite{MartinezRubioPokutta:2023:1} consider gradient schemes that combine approximate solutions to proximal maps defined on the manifold, leading to a globally accelerated first-order algorithm on Hadamard manifolds.
Furthermore,~\cite{MartinezRubioPokuttaRoux:2024:1} introduce a composite Riemannian descent method for geodesically convex functions that solves a proximal-type linearized subproblem at each iteration, which is also defined on the manifold.
Their convergence rates are in line with their Euclidean counterparts, albeit accounting for the curvature of the manifold.
After the first version of this manuscript was released on arXiv, the work in \cite{BentoSantiago:2026:1} was published.
The authors in said paper analyze the convergence properties of an intrinsic Riemannian Proximal Gradient method similar to ours.
The main difference to our work lies in their assumptions: they limit their analysis to Hadamard manifolds only, with the smooth component $g$ not assumed to be convex, whereas the nonsmooth component $h$ is assumed to be convex. 
Furthermore, they do not assume the gradient of $g$ to be Lipschitz.

\subsection{Contributions}
In this work, we introduce a formulation of the Riemannian proximal gradient method for (strongly) geodesically convex functions that actually employs a proximal map defined on the manifold for the nonsmooth step.
This continues our work on that method which was started in~\cite{BergmannJasaJohnPfeffer:2025:1}, where we investigate the nonconvex case.
It also relates the new proximal gradient method directly back to both the PPA~\cite{FerreiraOliveira:2002:1} as well as gradient descent, \eg~\cite{ZhangSra:2016:1}.
The analysis in this work is carried out on Riemannian manifolds of bounded sectional curvature, which can also be \emph{positive}.
We therefore generalize the findings of~\cite{FengHuangSongYingZeng:2021} while also using different assumptions, as well as investigate several properties from~\cite[Chapter 10]{Beck:2017:1}.

Most prominently:
\begin{itemize}
  \item \cref{thm:convex-convergence-rate} establishes the same sublinear convergence rate in the function values when the objective is assumed to be geodesically (not strongly) convex and it outlines a regime where this decay is actually linear.
  \item \cref{thm:convex-epsilon-stationary-pt} improves the convergence rate to $\varepsilon$-stationary points from $\cO \bigl(\frac{1}{\varepsilon^2}\bigr)$ in~\cite{BergmannJasaJohnPfeffer:2025:1} to $\cO \bigl(\frac{1}{\varepsilon}\bigr)$ under convexity assumptions.
  \item \cref{thm:mu-strongly-convex-convergence-rate} provides a linear convergence rate when strong geodesic convexity of the objective is assumed.
\end{itemize}
We also generalize the fundamental prox-grad inequality, which is extensively used in Beck's analysis, and showcase how this approach does \emph{not} directly provide convergence properties even in the simpler Hadamard case.
Finally, we provide open source code, as well as reproducible, extensive numerical experiments to test the performance of this method.
We show that in the examples we consider, the proximal map of $h$ either exists in closed-form or is efficiently computable intrinsically on the manifold.

\subsection{Organization}
The remainder of this paper is organized as follows.
After introducing the necessary notation and some preliminary properties in~\Cref{section:preliminaries}, we derive the new formulation of the Convex Riemannian Proximal Gradient (CRPG) method in \Cref{section:CRPG}.
In \Cref{section:Convergence}, we present convergence results for the geodesically convex and strongly convex cases.
In \Cref{section:proximal-gradient-inequalities}, we derive Riemannian variants of the fundamental prox-grad inequality that generalize~\cite[Theorem 10.16]{Beck:2017:1} and discuss their usefulness.
\Cref{section:numerics} illustrates the performance of the algorithm and its convergence, and \Cref{section:conclusion} concludes the paper.

\section{Preliminaries}%
\label{section:preliminaries}

In this section, we recall some notions from Riemannian geometry.
For more details, the reader may wish to consult~\cite{DoCarmo:1992:1,Boumal:2023:1}.
Let $\cM$ be a smooth manifold with Riemannian metric $\riemannian{\cdot}{\cdot}$ and let $\covariantDerivativeSymbol$ be its Levi-Civita connection.
Given $p, q \in \cM$, we denote by $C_{pq}$ the set of all piecewise smooth curves $\gamma \colon[0,1] \to \cM$ from~$p = \gamma(0)$ to~$q = \gamma(1)$.
The \emph{arc length} of $\gamma \in C_{pq}$, $L \colon C_{pq} \to \bbR_{\ge 0}$ is defined as
${
	L(\gamma)
	=
	\int_0^1 \riemanniannorm{\dot \gamma}[\gamma(t)] \, \d t
	,
}$
where $\riemanniannorm{\cdot}[p]$ is the norm induced by the Riemannian metric in the tangent space $\tangentSpace{p}$ at the point~$p$.
The inner product given by the Riemannian metric at~$p$ is ${\riemannian{\cdot}{\cdot}[p] \colon \tangentSpace{p} \times \tangentSpace{p} \to \bbR}$.
The explicit reference to the base point will be omitted whenever the base point is clear from context.
The sectional curvature at $p \in \cM$ is
$
	K_p(X_p,Y_p)
	=
	\frac{
		\riemannian{\mathrm{R}_p(X_p,Y_p) Y_p}{X_p}
	}{
		\riemanniannorm{X_p}^2
		\,
		\riemanniannorm{Y_p}^2
		-
		\riemannian{X_p}{Y_p}^2
	}
	,
$
where $\mathrm{R}$ is the Riemann curvature tensor, and $X_p, Y_p \in \tangentSpace{p}$ are linearly independent.

The Riemannian distance between two points, $\dist \colon \cM \times \cM \to \bbR_{\ge 0} \cup \{+\infty\}$, is given by
${
	\dist(p,q)
	=
	\inf_{\gamma \in C_{pq}} L(\gamma)
	.
}$
A \emph{geodesic arc} from $p \in \cM$ to $q \in \cM$ is a smooth curve $\gamma \in C_{pq}$ that is parallel along itself, \ie, $\covariantDerivative{\dot \gamma(t)}[\dot \gamma(t)] = 0$ for all $t \in (0,1)$.
A finite-length geodesic arc is \emph{minimal} if its arc length coincides with the Riemannian distance between its extreme points.
Given two points $p, q \in \cM$, we denote with $\geodesic{p}{q}$ a minimal geodesic arc that connects $p = \geodesic{p}{q}(0)$ to $q = \geodesic{p}{q}(1)$.
Given a point $p\in \cM$ and a tangent vector $X_p \in \tangentSpace{p}$, a geodesic can also be specified
as the curve $\geodesic{p}{X_p}$ such that $\geodesic{p}{X_p}(0)=p$ and $\geodesic{p}{X_p}'(0) = X_p$.
The \emph{exponential map} at $p \in \cM$ maps a tangent vector $X_p \in \tangentSpace{p}$ to the endpoint of the corresponding geodesic, $\exponential{p}(X_p) = \geodesic{p}{X_p}(1)$.
If it exists, the inverse map is called the \emph{logarithmic map} $\logarithm{p}\colon \cM \rightarrow \tangentSpace{p}$.
We call a space \emph{uniquely geodesic} if every two points in that space are connected by one and only one geodesic in that set.
In this case the logarithmic map exists everywhere.
A set $\cU \subseteq \cM$ is said to be \emph{(uniquely) geodesically convex} if for any $p,q \in \cU$, there exists a (unique) minimal geodesic arc $\geodesic{p}{q}$ which lies entirely in~$\cU$.
Observe that if a manifold $\cM$ has positive sectional curvature, the set $\cU$ may not be allowed to have arbitrarily large diameter.
For example, if $\kmin > 0$, for $\cU$ to be uniquely geodesically convex, its diameter must be bounded by $\diam(\cU) < \pi/\sqrt{\kmin}$.
See, \eg,~\cite[Proposition~4.2~(iii)]{LiYao:2012:1},~\cite[Proposition~4.1]{WangLiYao:2016:1} or~\cite[Proposition~II.1.4]{BridsonHaefliger:1999:1}.

We denote the \emph{parallel transport} from~$p$ to~$q$ along~$\geodesic{p}{q}$ with respect to the connection~$\covariantDerivativeSymbol$ by $\parallelTransport{p}{q} \colon \tangentSpace{p} \to \tangentSpace{q}$.
This is defined by
$
	\parallelTransport{p}{q} Y_p
	=
	X_q
$
for $Y_p \in \tangentSpace{p}$, where $X$ is the unique smooth vector field along $\geodesic{p}{q}$ which satisfies
$
	X_p
	=
	Y_p
	\text{ and }
	\covariantDerivative{\dot \gamma}{X}
	=
	0
	.
$
Note $\parallelTransport{p}{q}$ depends on the choice of $\geodesic{p}{q}$.

Given a differentiable function $f \colon \cM \to \bbR$, we denote the \emph{Riemannian gradient} of $f$ at $p \in \cM$ by $\grad f(p) \in \tangentSpace{p}$.
See~\cite[Definition~8.57]{Boumal:2023:1} for more details.
A function $f \colon \cM \to \bbR$ is said to be Lipschitz at $p \in \cM$ if there exists an open neighborhood $\cU$ of $p$ and a constant $L_p > 0$ such that
\begin{equation*}
    \abs{
        f(y)
        -
        f(z)
    }
    \le
    L_p
    \dist(y,z)
    \quad
    \text{for all }
    y
    ,
    z
    \in
    \cU
    .
\end{equation*}
Furthermore, $f$ is called locally Lipschitz if it is Lipschitz at every point $p \in \cM$.

For a geodesically convex set $\cU \subseteq \cM$, a proper function $f \colon \cU \to \eR$ and a point $p \in \dom(f)$, a tangent vector $X_p \in \tangentSpace{p}$ is said to be a \emph{subgradient} of $f$ at $p$ if $f(q) \ge f(p) + \riemannian{X_p}{\logarithm{p}{q}}$ for all $q \in \cU$.
The set of all subgradients of $f$ at $p$ is the \emph{subdifferential} of $f$ at $p$, and it is denoted by $\partial f(p)$.
This is a nonempty, convex, and compact set at any point in $\cM$.
Given a real number $\strongcvxconst > 0$, a function $f \colon \cU \to \eR$ is said to be \emph{($\strongcvxconst$-strongly) geodesically convex} over $\cU$ if, for any geodesic arc $\geodesic<s> \colon [0, 1] \to \cM$ contained in $\cU$, the composition $f \circ \geodesic<s>$ is ($\mu$-strongly with $\mu = \strongcvxconst L(\geodesic<s>)^2$) convex in the Euclidean sense; \cf~\cite[Definition~11.5]{Boumal:2023:1}.
This is equivalent to the following inequality holding for all $p,q \in \cU$ and all $X_p \in \partial f(p)$
\begin{equation}
    \label{eq:strong-convexity}
    f(q)
    \ge
    f(p)
    +
    \inner{
        X_p
    }{\logarithm{p}{q}}
    +
    \frac{\strongcvxconst}{2}
    \dist^2(p, q)
    .
\end{equation}
If the previous inequality is satisfied with $\strongcvxconst = 0$, the function $f$ is equivalently geodesically convex over $\cU$.
Furthermore, if $f$ is subdifferentiable and geodesically convex on an open geodesically convex set, $0_p \in \partial f(p)$ if and only if $p$ is a global minimizer of $f$ in $\cM$.
See, \eg,~\cite{Udriste:1994:1} for a thorough analysis of convex functions on Riemannian manifolds.
Given a real number $L_f > 0$, a differentiable function $f$ is said to be \emph{$L_f$-smooth} over $\cU$ if
\begin{equation*}
    f(q)
    \le
    f(p)
    +
    \inner{\grad f(p)}{\logarithm{p}{q}}
    +
    \frac{L_f}{2}
    \dist^2(p, q)
    ,
\end{equation*}
for any $p,q \in \cU$.
For a geodesically convex function $f$ and a real number $\lambda > 0$, define
\begin{equation*}
    \prox_{\lambda f}(p)
    \coloneq
    \argmin_{q \in \cM}
        f(q)
        +
        \frac{1}{2 \lambda} \dist^2(p, q)
    ,
\end{equation*}
if the minimizer exists.
If now $f = g + h$ with $g$ of class $\cC^1(\cU)$ and $h$ locally Lipschitz over $\cU$, then~\cite[Lemma~3.1]{BentoFerreiraOliveira:2015:1} gives
\begin{equation}
    \label{eq:clarke-subdifferential-of-sum}
    \partial (g + h) (p)
    =
    \grad g(p)
    +
    \partial h(p)
    ,
\end{equation}
for all $p \in \cU$.
Therefore, the condition $0_p \in \partial(g+h)(p) = \grad g(p) + \partial h(p)$ is equivalent to the stationarity condition
\begin{equation}
    \label{eq:stationarity-condition}
    - \grad g(p)
    \in
    \partial h(p)
    .
\end{equation}

In the analysis of the proximal gradient method, we will also need the following curvature-dependent quantities.
Given real numbers $\kappa_1, \kappa_2$, which will typically be lower and upper bounds to the sectional curvature of $\cM$, we define, for $s \in \bbR$,
\begin{equation}
	\label{eq:sigma-definition}
	\newtarget{def:sigma-curvature}{
        \sigma_{\kappa_1, \kappa_2}(s)
    }
    \coloneq
	\max \paren[big]\{\}{%
        \zeta_{1, \kappa_1}(s)
		,
		\;
        \abs{\zeta_{2, \kappa_2}(s)}
	}
	,
\end{equation}
where
\begin{equation}
	\label{eq:hessian-eigs}
	\begin{aligned}
		\newtarget{def:zeta-first}{
            \zeta_{1, \kappa_1}(s)
        }
        &
        \coloneq
        \begin{cases}
            1
            &
            \text{if }
            \kappa_1 \ge 0
            ,
            \\
            \sqrt{- \kappa_1} \, s \coth(\sqrt{- \kappa_1} \, s)
            \quad
            &
            \text{if }
            \kappa_1 < 0
            ,
        \end{cases}
		\\
		\newtarget{def:zeta-second}{
            \zeta_{2, \kappa_2}(s)
        }
        &
        \coloneq
        \begin{cases}
            1
            &
            \text{if }
            \kappa_2 \le 0
            ,
            \\
            \sqrt{\kappa_2} \, s \cot(\sqrt{\kappa_2} \, s)
            \quad
            &
            \text{if }
            \kappa_2 > 0
            ,
        \end{cases}
	\end{aligned}
\end{equation}
and
\begin{equation}
  \begin{aligned}
    \label{eq:generalized-sine-positive-zero}
    \newtarget{def:generalized-sine-positive-zeros}{
        \pikappa
    }
    &\coloneq
    \begin{cases}
        \infty
        &
        \text{if }
        \kappa \le 0
        ,
        \\
        \frac{\pi}{\sqrt{\kappa}}
        &
        \text{if }
        \kappa > 0
        .
    \end{cases}
  \end{aligned}
\end{equation}
For a given $\lambda > 0$ and $p \in \cM$, we will also often use the following notation for the gradient step throughout the manuscript
\begin{equation}
    \label{eq:gradient-step}
    \newtarget{def:gradient-step}{
        \gradientstep{p}
    }
    \coloneq
    \exponential{p}(-\lambda \grad g(p))
    .
\end{equation}
Furthermore, when referring to specific iterations of the algorithm, we will always assume $\gradientstep{\sequence{p}{n}}$ to be prescribed by the stepsize at the $n$-th step, that is, we write $\gradientstep{\sequence{p}{n}} = \exponential{\sequence{p}{n}}(-\sequence{\lambda}{n} \grad g(\sequence{p}{n}))$ for a given $n \in \bbN$, unless otherwise stated.
Finally, we will also denote the sublevel set of the objective $f$ at a point $\sequence{p}{0} \in \cM$ with
$
  \startlevelset
  \coloneq
  \setDef[auto]{p \in \cM}{
      f(p)
      \le
      f(\sequence{p}{0})
  }
  .
$

\section{The Riemannian Proximal Gradient method}%
\label{section:CRPG}

The main idea for the derivation of the proximal gradient method is to replace $g$ by a proximal regularization of its linearized function~\cite[Section~2.2]{BeckTeboulle:2009:1}.
For $\lambda > 0$ and a point $p$ we compute a new candidate as
\begin{equation*}
    \argmin_{q\in\cM}
    g(p)
    +
    \riemannian{\grad g(p)}{\logarithm{p}q}
    +
    \frac{1}{2\lambda}
    \dist^2(p, q)
    +
    h(q)
    .
\end{equation*}
As in~\cite[Section~3]{BergmannJasaJohnPfeffer:2025:1}, up to first order we obtain
\begin{equation*}
    \prox[big]{\lambda h}(
        \exponential{p}(-\lambda\grad g(p))
    )
    =
    \argmin_{q\in\cM}
    \frac{1}{2\lambda}
    \dist^2(
        \exponential{p}(-\lambda\grad g(p))
        ,
        q
    )
    +
    h(q)
    .
\end{equation*}
Note that this is defined entirely on the manifold, which is why we call it an \emph{intrinsic} method.
The formulation is similar to the one in~\cite{FengHuangSongYingZeng:2021}, where, in addition, an accelerated step is taken if it provides a decrease of the objective.
The final algorithm is summarized in \cref{algorithm:CRPG}.
\begin{algorithm}[htp]
	\caption{Convex Riemannian Proximal Gradient Method}%
	\label{algorithm:CRPG}
	\begin{algorithmic}[1]
		\Require
        $g$, %
        $\grad g$, %
        $h$, %
        $\prox_{\lambda h}$, %
        a sequence $\sequence{\lambda}{k}$, %
		an initial point $\sequence{p}{0} \in \cM$.
    \Ensure
    minimizer $\sequence{p}{k_\ast}$.
		\While{convergence criterion is not fulfilled}
		\State $
            \sequence{p}{k+1}
            =
            \prox_{\sequence{\lambda}{k} h}\bigl(
                \exponential{\sequence{p}{k}}(-\sequence{\lambda}{k} \grad g(\sequence{p}{k}))
            \bigr)
		$
		\State Set $k \coloneq k+1$
		\EndWhile
	\end{algorithmic}
\end{algorithm}
We emphasize that, in the examples we consider, the proximal map of $h$ is either available in closed-form or can be efficiently computed on the manifold.
See \cref{section:numerics} for more details.

\section{Convergence}%
\label{section:Convergence}

In this section, we analyze the convergence behavior of the Convex Riemannian Proximal Gradient (CRPG) method.
When $g$ is assumed to be geodesically convex, we obtain a sublinear convergence rate of the cost function.
Furthermore, an improved convergence rate to $\varepsilon$-stationary points as compared to the nonconvex case (NCRPG) analyzed in~\cite{BergmannJasaJohnPfeffer:2025:1} is proven.
In case $g$ is strongly geodesically convex, we can show a linear convergence rate.

Let us introduce the standing set of assumptions that we will use in the rest of this section.
\begin{assumption}[Standing assumptions] \hfill%
	\label{assumption:assumptions}
	\begin{enumerate}
    \item\label[item]{assumption:assumptions:bounded-sectional-curvature}
      There exist two numbers $\kmin, \kmax \in \bbR$ such that the sectional curvature satisfies $\newtarget{def:k-min}{\kmin} \le K_p(X_p,Y_p) \le \newtarget{def:k-max}{\kmax}$ for all $p \in \cM$ and all linearly independent tangent vectors $X_p, Y_p \in \tangentSpace{p}$.

		\item\label[item]{assumption:assumptions:geodesically-convex-function}
      $h \colon \cM \to \eR$ is proper, closed, and geodesically convex over the uniquely geodesically convex set $\dom (h)$.

		\item\label[item]{assumption:assumptions:smooth-function}
      $g \colon \cM \to \eR$ is proper and closed, $\dom(g)$ is uniquely geodesically convex, $\dom(h) \subseteq \interior \dom(g)$, and $g$ is $\newtarget{def:gradg-lipschitz-constant}{\lipgrad}$-smooth over $\interior \dom(g)$.

    \item\label[item]{assumption:assumptions:nonempty-optimal-set}
      It holds $\cS \coloneq \argmin_{p \in \cM} f(p) \neq \emptyset$ and we define $f_{\text{opt}} \coloneq f(p)$ for $p \in \cS$.
    \item\label[item]{assumption:bounded-level-sets}
      There exists a number $c \in \bbR$ such that $\{p \in \dom(f) \mid f(p) \le c\}$ is bounded.
  \end{enumerate}
\end{assumption}
Under these assumptions, we introduce the \emph{iteration map} $\imapOp[auto]{g}{h}{\lambda} \colon \cM \to \cM$
\begin{equation*}
    \imap[auto]{g}{h}{\lambda}{p}
    \coloneq
    \prox{\lambda h}\Bigl(
        \exponential[big]{p}(-\lambda \grad g(p))
    \Bigr)
    ,
\end{equation*}
which is well-defined on $\dom(h)$ since the function $h$ is geodesically convex over this set, making the proximal map single-valued.
As $g, h$ are fixed in the following, we will use $\Imap[auto]{\lambda}{p}$ as a shorthand notation for the iteration map.
\begin{remark}%
  \label{remark:compact-level-sets}
  Note that by \cref{assumption:assumptions:geodesically-convex-function}, \cref{assumption:assumptions:smooth-function}, and \cref{assumption:bounded-level-sets} of \cref{assumption:assumptions}, the sublevel set $\cL_{\sequence{p}{0}}$ is compact for every $\sequence{p}{0} \in \dom(f)$.
  This also implies that for each $\sequence{p}{0}$, there exist numbers $\hlowerbound, \hupperbound, \gradgupperbound \in \bbR$ such that $\hlowerbound \le h(p) \le \hupperbound$, and $\riemanniannorm{\grad g(p)} \le \gradgupperbound$ for all $p \in \cL_{\sequence{p}{0}}$.
\end{remark}

In this work, the analysis requires that the current iterate, the gradient step, and the next iterate be within the injectivity radius of $\cM$ at the current iterate.
This is achieved by bounding the stepsize according to~\cite[Lemma~4.3]{BergmannJasaJohnPfeffer:2025:1} via
\begin{equation}
  \label{eq:lambda-delta}
  \newtarget{def:delta-stepsize}{
      \lambdadelta
  }
  \coloneq
  \frac{
          \sqrt{
              4 \, (\hupperbound - \hlowerbound)^2
              +
              \frac{{\pikmax}^2}{(2 + \delta)^2} \, \gradgupperbound^2
          }
          -
          2 \, (\hupperbound - \hlowerbound)
  }{2 \, \gradgupperbound^2}
  ,
\end{equation}
for some $\delta > 0$.
Then we have for all $p \in \cL_{\sequence{p}{0}}$ and $\lambda \le \lambdadelta$
\begin{equation}
    \label{eq:boundedness-of-iteration}
    \max\{
        \dist(p, \gradientstep{p})
        ,
        \dist(p,\Imap{\lambda}{p})
    \}
    \le
    \deltaradius,
\end{equation}
where $\deltaradius \coloneq \frac{\pikmax}{2+\delta}$.
Note that, if $\kmax \le 0$, $\lambda_\delta \equiv + \infty$ for all $\delta > 0$.
The same lemma yields
\begin{equation}
  \label{ineq:zetadelta}
  \zetasecond{\dist(p, q)}
  \ge
  \zeta_\delta
  >
  0
  ,
\end{equation}
for all points $q \in \cB(p, \deltaradius)$, with $\cB(p, \deltaradius)$ being the open ball around $p \in \dom(f)$ of radius $\deltaradius$, where
\begin{equation}
  \label{eq:zeta-delta}
  \newtarget{def:delta-curvature}{
      \zetadelta
  }
  \coloneq
  \begin{cases}
    1
    &
    \text{if }
    \kmax \le 0
    ,
    \\
    \frac{\pi}{2 + \delta} \cot \left(\frac{\pi}{2 + \delta} \right)
    \quad
    &
    \text{if }
    \kmax > 0
    ,
  \end{cases}
\end{equation}
which we will use in the stepsize selection.

\begin{remark}%
    \label{remark:stepsizes}
    We will consider a constant stepsize and a backtracking procedure.
    For the constant case, choose $\sequence{\lambda}{k} = \lambda \in \left( 0, \min\{\lambdadelta, \frac{\zetadelta}{\lipgrad} \} \right)$ for a given $\delta > 0$.
    The backtracking procedure is given in~\cref{algorithm:backtracking}.
    The warm start procedure in~\cref{line:warm-start} ensures that every stepsize $\sequence{\lambda}{k}$ is bounded by the initial guess $s$, while giving the possibility to increase the stepsize by a warm start factor of $\theta$ to try to prevent the stepsize from collapsing, \eg due to numerical issues.
    Note that the first condition ensures that \cref{eq:boundedness-of-iteration} holds, allowing the backtracking procedure to be applied without further knowledge of the cost function. Furthermore, \cref{eq:boundedness-of-iteration} is automatically satisfied if the stepsize $\sequence{\lambda}{k}$ falls below $\lambdadelta$.
    This, together with \cref{lemma:convex-stepsize-bounds} below, implies that \cref{algorithm:backtracking} terminates in a finite number of iterations.
    \begin{algorithm}[htp]
        \caption{Backtracking procedure for stepsize selection
        }%
        \label{algorithm:backtracking}
        \begin{algorithmic}[1]
            \Require
            $f$,
            $k \ge 0$,
            last stepsize $\sequence{\lambda}{k-1}$,
            iterate $\sequence{p}{k}$,
            initial guess $s >0$,
            constants $\eta \in (0,1)$,
            $\theta \ge 1$.
            \Ensure
            stepsize $\sequence{\lambda}{k}$.
            \State\label{line:warm-start}
            Set $\sequence{\lambda}{-1} \coloneq s$, and $\sequence{\lambda}{k} = \min\{s, \theta \sequence{\lambda}{k-1}\}$
            \While {true}
                \If{$\max\{\dist(\sequence{p}{k}, \gradientstep{\sequence{p}{k}}), \dist(\sequence{p}{k},\Imap{\sequence{\lambda}{k}}{\sequence{p}{k}})\} \le \deltaradius$}
                    \If{$
                        g\left(
                        \Imap[auto]{\sequence{\lambda}{k}}{\sequence{p}{k}}
                        \right)
                        \le
                        g(\sequence{p}{k})
                        +
                        \riemannian[auto]{
                            \grad g(\sequence{p}{k})
                        }{
                            \logarithm{\sequence{p}{k}}{
                                \Imap[auto]{\sequence{\lambda}{k}}{\sequence{p}{k}}
                            }
                        }
                        $
                        \\
                        $
                        \qquad
                        \qquad
                        \qquad
                        \qquad
                        \qquad
                        \qquad
                        +
                        \frac{\zetadelta}{2 \sequence{\lambda}{k}}
                        \dist^2\left(
                            \sequence{p}{k}
                            ,
                            \Imap[auto]{\sequence{\lambda}{k}}{\sequence{p}{k}}
                        \right)
                        $}
                    \State \textbf{break}
                \EndIf
                \EndIf
                \State Set $\sequence{\lambda}{k} = \eta \sequence{\lambda}{k}$
            \EndWhile
        \end{algorithmic}
    \end{algorithm}
\end{remark}
The following lemma establishes upper and lower bounds on the stepsize in the case where $h$ is geodesically convex, and is a direct generalization of~\cite[Remark~10.19]{Beck:2017:1}.
The proof is therefore omitted.
\begin{lemma}%
  \label{lemma:convex-stepsize-bounds}
  If \cref{assumption:assumptions} holds, then both stepsize rules discussed in \cref{remark:stepsizes} are such that the sufficient decrease condition for $g$, given by
  \begin{equation}
      \label{eq:convex-sufficient-decrease-condition}
      g(\sequence{p}{k+1})
      \le
      g(\sequence{p}{k})
      +
      \riemannian{
          \grad g(\sequence{p}{k})
      }
      {
          \logarithm{\sequence{p}{k}}{\sequence{p}{k+1}}
      }
      +
      \frac{\zetadelta}{2 \sequence{\lambda}{k}}
      \dist^2(
          \sequence{p}{k}
          ,
          \sequence{p}{k+1}
      )
      ,
  \end{equation}
  is satisfied for all $k \in \bbN$.
  Furthermore, the following bounds hold for all the stepsizes $\sequence{\lambda}{k}$ returned by the backtracking procedure
  \begin{equation}
    \min\left\{
        s,
        \frac{\eta \zetadelta}{\lipgrad}
        , \eta \lambdadelta
        \right\}
      \le
      \sequence{\lambda}{k}
      \le
      s
      .
  \end{equation}
  This is equivalent to
  \begin{equation}
      \label{eq:convex-stepsize-bounds}
      \frac{\betastepsize}{\lipgrad}
      \le
      \sequence{\lambda}{k}
      \le
      \frac{\alphastepsize}{\lipgrad}
      ,
  \end{equation}
  for all $k \ge 0$, where
  \begin{equation*}
      \newtarget{def:alpha-stepsize}{
          \alphastepsize
      }
      =
      \begin{cases}
          \lambda \lipgrad & \text{ constant,} \\
          s \lipgrad & \text{ backtracking,}
      \end{cases}
      ,
      \quad
      \newtarget{def:beta-stepsize}{
          \betastepsize
      }
      =
      \begin{cases}
          \lambda \lipgrad & \text{ constant,} \\
          \min\{
            s \lipgrad,
            \eta\zetadelta
            , \eta \lambdadelta
            \}
            & \text{ backtracking,}
      \end{cases}
  \end{equation*}
  with $\lambda$ chosen as in \cref{remark:stepsizes}.
\end{lemma}

\subsection{Convex Objective}

In~\cite{BergmannJasaJohnPfeffer:2025:1}, we have investigated the RPG algorithm for general Riemannian manifolds of bounded sectional curvature and nonconvex $g$ and $h$.
We showed that the algorithm reaches $\varepsilon$-stationary points at a sublinear rate of $\cO \bigl(\frac{1}{\varepsilon^2}\bigr)$.
Furthermore, all accumulation points of the series of iterates are stationary points.
These results clearly hold for the convex setting in particular.
Additionally, we can now improve the rate for obtaining $\varepsilon$-stationary points to $\cO \bigl(\frac{1}{\varepsilon}\bigr)$ in the convex case and show sublinear convergence of the cost function with a rate of $\cO \bigl(\frac{1}{k}\bigr)$.

We start by adding the following assumption to \cref{assumption:assumptions}
\begin{assumption}[Convexity] \hfill%
	\label{assumption:convex-assumptions}
	\begin{itemize}
    \item $g$ is geodesically convex over $\interior\dom(g)$.
    \end{itemize}
\end{assumption}
We can now state our main convergence result of this section, noting that the prox-grad inequalities do not provide a successful strategy to do so, as argued in \Cref{section:proximal-gradient-inequalities}.
\begin{theorem}%
    \label{thm:convex-convergence-rate}
    Let $\sequence{\Delta}{k} \coloneq f(\sequence{p}{k}) - f_{\text{opt}}$ for $k \ge 0$.
    Under the validity of \cref{assumption:assumptions} and \cref{assumption:convex-assumptions}, the sequence $\sequence(){p}{k}$ generated by \cref{algorithm:CRPG} with either of the stepsize strategies stated in \cref{remark:stepsizes} satisfies the following rates
    \begin{itemize}
        \item if at iteration $k-1$, it holds that $\frac{\sequence{\lambda}{k-1} \sequence{\Delta}{k-1}}{\zeta_{1, \kmin}(\sequence{D}{k-1}) \dist^2(\sequence{p}{k-1}, p^\ast)} \ge 1$, then
        \begin{equation}
            \label{eq:linear-convex-convergence-rate}
            \sequence{\Delta}{k}
            \le
            \frac{1}{2}
            \sequence{\Delta}{k-1}
            ;
        \end{equation}
        \item otherwise,
        \begin{equation}
            \label{eq:convex-convergence-rate-2}
            \sequence{\Delta}{k}
            \le
            \frac{\sequence{\Delta}{0}}{1 + k \, \tau \sequence{\Delta}{0}}
            ,
        \end{equation}
    \end{itemize}
    for all $p^\ast \in \cS$ and $k \ge 0$, where
    $
      \tau
      \coloneq
      \min\left\{
        \frac{\betastepsize}{2 \lipgrad \zetafirst{\frac{\alphastepsize}{\lipgrad} \gradgupperbound} R^2},
          \frac{1}{\sequence{\Delta}{0}}
      \right\}
      ,
      $ with $\gradgupperbound$ as in \cref{remark:compact-level-sets}, $\alphastepsize, \betastepsize$ as in \cref{lemma:convex-stepsize-bounds}, $R \coloneq \diam(\startlevelset)$, and $\sequence{D}{n} \coloneq \dist(\sequence{p}{n}, \gradientstep{\sequence{p}{n}})$.
    In particular, for any $\varepsilon > 0$, $\sequence{\Delta}{k} \le \varepsilon$ is achieved as soon as
    \begin{equation}
        \label{eq:convex-complexity}
        k
        \ge
        \max
        \left\{
            \left\lceil
                \log_2
                \left(
                    \frac{\sequence{\Delta}{0}}{\varepsilon}
                \right)
            \right\rceil
            ,
            \left\lceil
                \frac{
                    \sequence{\Delta}{0}
                    -
                    \varepsilon
                }{
                    \tau
                    \varepsilon
                    \sequence{\Delta}{0}
                }
            \right\rceil
        \right\}
        .
    \end{equation}
\end{theorem}
\begin{proof}
    Let $\sequence{p}{n} \in \cM$ and $\sequence{\lambda}{n} > 0$ be such that \cref{eq:convex-sufficient-decrease-condition} is satisfied.
    We have
    \begin{equation}
        \label{eq:cost-splitting}
        f(\sequence{p}{n+1})
        =
        g(\sequence{p}{n+1})
        +
        h(\sequence{p}{n+1})
        ,
    \end{equation}
    where $\sequence{p}{n+1}$ is the argmin of
    $
        h(p)
        +
        \frac{1}{2\sequence{\lambda}{n}}
        \dist^2(p, \gradientstep{\sequence{p}{n}})
        .
    $
    By \cref{eq:convex-sufficient-decrease-condition},
    \begin{align*}
        g(\sequence{p}{n+1})
        &
        \le
        g(\sequence{p}{n})
        +
        \riemannian{
            \grad g(\sequence{p}{n})
        }{
            \logarithm{\sequence{p}{n}}{\sequence{p}{n+1}}
        }
        +
        \frac{\zetadelta}{2\sequence{\lambda}{n}}
        \dist^2(\sequence{p}{n+1}, \sequence{p}{n})
        \\
        &
        =
        g(\sequence{p}{n})
        -
        \frac{1}{\sequence{\lambda}{n}} \riemannian{
            \logarithm{\sequence{p}{n}}{\gradientstep{\sequence{p}{n}}}
        }{
            \logarithm{\sequence{p}{n}}{\sequence{p}{n+1}}
        }
        +
        \frac{\zetadelta}{2\sequence{\lambda}{n}}
        \dist^2(\sequence{p}{n+1}, \sequence{p}{n})
        .
    \end{align*}
    By the Riemannian cosine inequality~\cite[Corollary~2.1]{AlimisisBecigneulLucchiOrvieto:2020:1}:
    \begin{multline*}
        -
        2 \riemannian{
          \logarithm{\sequence{p}{n}}{\gradientstep{\sequence{p}{n}}}
        }{
          \logarithm{\sequence{p}{n}}{\sequence{p}{n+1}}
        }
        \le
        -
        \zeta_{2, \kmax}(\dist(m, \sequence{p}{n}))
        \dist^2(\sequence{p}{n+1}, \sequence{p}{n})
        \\
        -
        \dist^2(\sequence{p}{n}, \gradientstep{\sequence{p}{n}})
        +
        \dist^2(\sequence{p}{n+1}, \gradientstep{\sequence{p}{n}})
        ,
    \end{multline*}
    for some point $m$ on the edge between $\sequence{p}{n}$ and $\sequence{p}{n+1}$.
    Using both these inequalities in \cref{eq:cost-splitting} we get
    \begin{multline*}
        f(\sequence{p}{n+1})
        \le
        g(\sequence{p}{n})
        -
        \frac{1}{2\sequence{\lambda}{n}}\dist^2(\sequence{p}{n}, \gradientstep{\sequence{p}{n}})
        \\
        +
        \frac{\zetadelta - \zetasecond{\dist(m, \sequence{p}{n})}}{2\sequence{\lambda}{n}}
        \dist^2(\sequence{p}{n+1}, \sequence{p}{n})
        \\
        +
        \frac{1}{2\sequence{\lambda}{n}}
        \dist^2(\sequence{p}{n+1}, \gradientstep{\sequence{p}{n}})
        +
        h(\sequence{p}{n+1})
        .
    \end{multline*}
    Since $\dist(\sequence{p}{n}, \sequence{p}{n+1})< \deltaradius$ is ensured by both stepsize rules for all $k$, we observe that ${\zetadelta - \zetasecond{\dist(m, \sequence{p}{n})} \le 0}$ by \cref{ineq:zetadelta}, and we can hence drop this term in the previous inequality.
    By the geodesic convexity of $g$ and by the definition of $\sequence{p}{n+1}$ applied to the last two summands,
    \begin{align*}
        f(\sequence{p}{n+1})
        \le
        \min_{p \in \cM}
        \Bigl\{
            g(p)
            &
+
            \frac{1}{\sequence{\lambda}{n}}
            \riemannian{
                \logarithm{\sequence{p}{n}}{\gradientstep{\sequence{p}{n}}}
            }{
                \logarithm{\sequence{p}{n}}{p}
            }
            -
            \frac{1}{2\sequence{\lambda}{n}}
            \dist^2(\sequence{p}{n}, \gradientstep{\sequence{p}{n}})
            \\
            &
            +
            h(p)
            +
            \frac{1}{2\sequence{\lambda}{n}}\dist^2(p, \gradientstep{\sequence{p}{n}})
        \Bigr\}
        .
    \end{align*}
    Let $\sequence{D}{n} \coloneq \dist(\sequence{p}{n}, \gradientstep{\sequence{p}{n}})$.
    By~\cite[Remark~20]{MartinezRubioPokuttaRoux:2024:1}, or~\cite[Lemma~6,~Remark~7]{ZhangSra:2016:1}, we can use the other Riemannian cosine inequality with the tighter constant $\zeta_{1, \kmin}(\sequence{D}{n})$ to obtain
    \begin{align*}
        2
        \riemannian{
            \logarithm{\sequence{p}{n}}{\gradientstep{\sequence{p}{n}}}
        }{
            \logarithm{\sequence{p}{n}}{p}
        }
        &
        \le
        \zeta_{1, \kmin}(\sequence{D}{n})
        \dist^2(\sequence{p}{n}, p)
        +
        \dist^2(\sequence{p}{n}, \gradientstep{\sequence{p}{n}})
        -
        \dist^2(p, \gradientstep{\sequence{p}{n}})
        ,
    \end{align*}
    which substituted back into the previous inequality gives
    \begin{equation*}
        f(\sequence{p}{n+1})
        \le
        \min_{p \in \cM}\left\{
            f(p)
            +
            \frac{\zeta_{1, \kmin}(\sequence{D}{n})}{2\sequence{\lambda}{n}}
            \dist^2(\sequence{p}{n}, p)
        \right\}
        .
    \end{equation*}

    Inspired by~\cite[Proposition~22]{MartinezRubioPokuttaRoux:2025:1}, let $\theta \in [0, 1]$, so that $p = \exponential{\sequence{p}{n}}(\theta \logarithm{\sequence{p}{n}}{p^\ast})$.
    By the geodesic convexity of $f$ we get, by restricting the min to the geodesic between $p^\ast$ and $\sequence{p}{n}$,
    \begin{equation}
        \label{eq:min-theta-convex-geodesic-convexity}
        f(\sequence{p}{n+1})
        \le
        \min_{\theta \in [0,1]}\left\{
            \theta f(p^\ast)
            +
            (1 - \theta)f(\sequence{p}{n})
            +
            \frac{\zeta_{1, \kmin}(\sequence{D}{n}) \theta^2}{2\sequence{\lambda}{n}}
            \dist^2(\sequence{p}{n}, p^\ast)
        \right\}
        .
    \end{equation}
    Differentiating the argument of the minimum in the previous expression with respect to $\theta$ and setting it to zero yields $\theta^\ast = \min\left\{1, \frac{\sequence{\lambda}{n} \sequence{\Delta}{n}}{\zeta_{1, \kmin}(\sequence{D}{n}) \dist^2(\sequence{p}{n}, p^\ast)}\right\}$.

    If $\theta^\ast = 1$, we obtain a \enquote{best case} convergence rate that is independent of curvature.
    Substituting $\theta^\ast = 1$ in \cref{eq:min-theta-convex-geodesic-convexity} gives
    \begin{equation}
        \label{eq:best-case-convex-convergence}
        f(\sequence{p}{n+1})
        \le
        f(p^\ast)
        +
        \frac{\zeta_{1, \kmin}(\sequence{D}{n})}{2 \sequence{\lambda}{n}}
        \dist^2(\sequence{p}{n}, p^\ast)
        .
    \end{equation}
    Since $\theta^\ast = 1$, $\frac{\sequence{\lambda}{n} \sequence{\Delta}{n}}{\zeta_{1, \kmin}(\sequence{D}{n}) \dist^2(\sequence{p}{n}, p^\ast)} \ge 1$, which implies $\sequence{\Delta}{n} \ge \frac{\zeta_{1, \kmin}(\sequence{D}{n}) \dist^2(\sequence{p}{n}, p^\ast)}{\sequence{\lambda}{n}}$.
    Subtracting $f(p^\ast)$ from both sides in \cref{eq:best-case-convex-convergence} and plugging in the previous inequality yields \cref{eq:linear-convex-convergence-rate}.
    In particular, this implies that, for any $\varepsilon > 0$, $\sequence{\Delta}{k} \le \varepsilon$ for all
    \begin{equation}
        \label{eq:best-case-convex-complexity}
        k
        \ge
        \log_2
        \left(
            \frac{\sequence{\Delta}{0}}{\varepsilon}
        \right)
        ,
    \end{equation}
    in case $\frac{\sequence{\lambda}{n} \sequence{\Delta}{n}}{\zeta_{1, \kmin}(\sequence{D}{n}) \dist^2(\sequence{p}{n}, p^\ast)} \ge 1$ for all $n < k$.

    Otherwise, $\theta^\ast = \frac{\sequence{\lambda}{n} \sequence{\Delta}{n}}{\zeta_{1, \kmin}(\sequence{D}{n}) \dist^2(\sequence{p}{n}, p^\ast)} < 1$.
    Substituting $\theta^\ast$ in \cref{eq:min-theta-convex-geodesic-convexity} yields
    \begin{equation}
        \label{eq:worst-case-convex-convergence-bound-1}
        f(\sequence{p}{n+1})
        \le
        - \frac{
            \sequence{\lambda}{n}
        }{
            2 \zeta_{1, \kmin}(\sequence{D}{n}) \dist^2(\sequence{p}{n}, p^\ast)
        } \left[\sequence{\Delta}{n}\right]^2
        +
        f(\sequence{p}{n})
        .
    \end{equation}
    We observe that $\sequence{D}{n} \le \frac{\alphastepsize}{\lipgrad} \gradgupperbound$ for all $n = 0, \ldots, k+1$, and $R \coloneq \diam(\startlevelset) < +\infty$ by \cref{remark:compact-level-sets}
    Applying these and \cref{eq:convex-stepsize-bounds} in \cref{eq:worst-case-convex-convergence-bound-1} and rearranging the terms gives
    \begin{equation}
        \label{eq:worst-case-convex-convergence-bound-2}
        \sequence{\Delta}{n+1}
        \le
        \sequence{\Delta}{n}
        -
        \tau
        \left[\sequence{\Delta}{n}\right]^2
        ,
    \end{equation}
    with $
      \tau
      \coloneq
      \min\left\{
          \frac{\betastepsize}{2 \lipgrad \zetafirst{\frac{\alphastepsize}{\lipgrad} \gradgupperbound} R^2}
          ,
          \frac{1}{\sequence{\Delta}{0}}
      \right\}
    $.
    We now prove by induction that
    \begin{equation}
        \label{eq:worst-case-convex-convergence-bound-3}
        \frac{1}{\sequence{\Delta}{n}}
        \ge
        \frac{1}{\sequence{\Delta}{0}}
        +
        n \tau
        ,
    \end{equation}
    for all $n = 0, \ldots, k+1$.
    The base case $n = 0$ is trivial.
    If $\theta^*<1$, then dividing \cref{eq:worst-case-convex-convergence-bound-2} by $\sequence{\Delta}{n} \sequence{\Delta}{n+1} \neq 0$ and rearranging the terms gives
    $
        \frac{1}{\sequence{\Delta}{n+1}}
        \ge
        \frac{1}{\sequence{\Delta}{n}}
        +
        \tau
        ,
    $
    where we also used $\frac{\sequence{\Delta}{n}}{\sequence{\Delta}{n+1}} \ge 1$, which is another consequence of \cref{eq:worst-case-convex-convergence-bound-2}.
    If $\theta^* = 1$, then by \cref{eq:linear-convex-convergence-rate},
    $
        \frac{1}{\sequence{\Delta}{n+1}}
        \ge
        \frac{1}{\sequence{\Delta}{n}}
        +
        \frac{1}{\sequence{\Delta}{n}}
        \ge
        \frac{1}{\sequence{\Delta}{n}}
        +
        \frac{1}{\sequence{\Delta}{0}}
        \ge
        \frac{1}{\sequence{\Delta}{n}}
        +
        \tau
    $.
    Assume that \cref{eq:worst-case-convex-convergence-bound-3} holds for some $n \ge 0$ and plug this into the previous inequality.
    Then we get
    ${
        \frac{1}{\sequence{\Delta}{n+1}}
        \ge
        \frac{1}{\sequence{\Delta}{0}}
        +
        n \tau
        +
        \tau
        =
        \frac{1}{\sequence{\Delta}{0}}
        +
        (n+1) \tau
        ,
    }$
    which completes the induction proof.
    \cref{eq:convex-convergence-rate-2} follows.

    From this we are able to derive a worst-case complexity bound.
    Let $\varepsilon > 0$ and suppose $\theta^\ast < 1$ for all $n < k$.
    Then, $\sequence{\Delta}{k} \le \varepsilon$ if $\frac{\sequence{\Delta}{0}}{\varepsilon} \le 1 + k \tau \sequence{\Delta}{0}$, which holds if and only if $k \ge \frac{\sequence{\Delta}{0} - \varepsilon}{\varepsilon \tau \sequence{\Delta}{0}}$.
    This, together with \cref{eq:best-case-convex-complexity} yields \cref{eq:convex-complexity}.
\end{proof}

\begin{remark*}
The previous theorem shows that the shrinkage of the function values is characterized by two regimes, depending on the ratio $\frac{\sequence{\lambda}{k} \sequence{\Delta}{k}}{\zeta_{1, \kmin}(\sequence{D}{k}) \dist^2(\sequence{p}{k}, p^\ast)}$ at iteration $k+1$.
Namely, if this ratio is larger than one, the rate is linear, whereas if it is strictly smaller than one, the rate is sublinear.
This would seem to suggest that, in the earlier iterations, \ie when $\sequence{\Delta}{k}$ is large, the algorithm converges faster than in the later iterations, when $\sequence{\Delta}{k}$ is small.
This is in contrast to what is often observed in practice, where $\dist^2(\sequence{p}{k}, p^\ast)$ decreases faster than $\sequence{\Delta}{k}$ near the solution, making the ratio larger than one and thus leading to a linear convergence rate.
\end{remark*}

\subsubsection{Convergence to $\varepsilon$-stationary points}

We can improve our previous result from\ \cite[Lemma~4.4]{BergmannJasaJohnPfeffer:2025:1} whenever the (possibly) nonsmooth function $h$ is assumed to be geodesically convex.
\begin{lemma}
  \label{lemma:convex-stepsize-sufficient-decrease}
  If \cref{assumption:assumptions} holds, then with $\sequence{\lambda}{k}$ chosen according to one of the stepsize rules discussed in \cref{remark:stepsizes}, the following sufficient decrease inequality holds at each iteration $k$
  \begin{equation*}
    f(\sequence{p}{n})
    -
    f(\sequence{p}{n+1})
    \ge
    \frac{\zetadelta}{2 \sequence{\lambda}{k}}
    \dist^2(\sequence{p}{n}, \sequence{p}{n+1})
    .
  \end{equation*}
\end{lemma}
\begin{proof}
  By the sufficient decrease condition of $g$ \cref{eq:convex-sufficient-decrease-condition} and the convexity of $h$ over $\dom(h)$, we have
  \begin{align*}
    f(\sequence{p}{n})
    &
    \ge
    g(\sequence{p}{n+1})
    +
    \frac{1}{\sequence{\lambda}{k}}
    \riemannian[auto]{
        \logarithm{\sequence{p}{n}}{\gradientstep{\sequence{p}{n}}}
    }{
        \logarithm{\sequence{p}{n}}{\sequence{p}{n+1}}
    }
    -
    \frac{\zetadelta}{2 \sequence{\lambda}{k}}
    \dist^2(\sequence{p}{n},\sequence{p}{n+1})
    \\
    &
    \quad
    +
    h(\sequence{p}{n+1})
    +
    \frac{1}{\sequence{\lambda}{k}}
    \riemannian[auto]{
        \logarithm{\sequence{p}{n+1}}{\gradientstep{\sequence{p}{n}}}
    }{
        \logarithm{\sequence{p}{n+1}}{\sequence{p}{n}}
    }
    ,
  \end{align*}
  where we also used the definition of $\gradientstep{\sequence{p}{n}}$ and $\frac{1}{\sequence{\lambda}{k}}\logarithm{\sequence{p}{n+1}}{\gradientstep{\sequence{p}{n}}} \in \partial h(\sequence{p}{n+1})$.
  Applying the Riemannian cosine inequality in~\cite[Corollary~15]{MartinezRubioPokutta:2023:1} to both inner product terms we get
  \begin{align*}
      \frac{1}{\sequence{\lambda}{k}}
      \riemannian[auto]{
          \logarithm{\sequence{p}{n}}{\gradientstep{\sequence{p}{n}}}
      }{
          \logarithm{\sequence{p}{n}}{\sequence{p}{n+1}}
      }
      &
      \ge
      \frac{1}{2 \sequence{\lambda}{k}}
      \dist^2(\sequence{p}{n}, \gradientstep{\sequence{p}{n}})
      -
      \frac{1}{2 \sequence{\lambda}{k}}
      \dist^2(\gradientstep{\sequence{p}{n}}, \sequence{p}{n+1})
      \\
      &
      \quad
      +
      \frac{\zetasecond{\dist(\sequence{p}{n}, m_1)}}{2 \sequence{\lambda}{k}}
      \dist^2(\sequence{p}{n}, \sequence{p}{n+1})
      ,
  \end{align*}
  where $m_1$ is a point on the minimal geodesic arc connecting $\sequence{p}{n+1}$ and $\gradientstep{\sequence{p}{n}}$, and
  \begin{align*}
    \frac{1}{\sequence{\lambda}{k}}
    \riemannian{
      \logarithm{\sequence{p}{n+1}}{\gradientstep{\sequence{p}{n}}}
    }{
      \logarithm{\sequence{p}{n+1}}{\sequence{p}{n}}
    }
    &
    \ge
    \frac{1}{2 \sequence{\lambda}{k}}
    \dist^2(\sequence{p}{n+1}, \gradientstep{p})
    -
    \frac{1}{2 \sequence{\lambda}{k}}
    \dist^2(\gradientstep{\sequence{p}{n}}, \sequence{p}{n})
    \\
    &
    \quad
    +
    \frac{\zetasecond{\dist(\sequence{p}{n+1}, m_2)}}{2 \sequence{\lambda}{k}}
    \dist^2(\sequence{p}{n+1}, \sequence{p}{n})
    ,
  \end{align*}
  where $m_2$ is a point on the minimal geodesic arc connecting $\sequence{p}{n}$ and $\gradientstep{\sequence{p}{n}}$.
  Using \cref{ineq:zetadelta} we get $\zetasecond{\dist(m_1, \sequence{p}{n})} \ge \zetadelta$.
  A similar reasoning applies to the second inequality.
  All together this gives
  \begin{align*}
    f(\sequence{p}{n})
    &
    \ge
    f(\sequence{p}{n+1})
    -
    \frac{\zetadelta}{2 \sequence{\lambda}{k}}
    \dist^2(\sequence{p}{n}, \sequence{p}{n+1})
    +
    \frac{\zetadelta}{\sequence{\lambda}{k}}
    \dist^2(\sequence{p}{n}, \sequence{p}{n+1})
    ,
  \end{align*}
  which implies the claim.
\end{proof}
Next, we show how the convergence of \cref{algorithm:CRPG} to $\varepsilon$-stationary points can be improved from $\mathcal O\bigl(\frac{1}{\varepsilon^2}\bigr)$ in the nonconvex case to $\mathcal O\bigl(\frac{1}{\varepsilon}\bigr)$ iterations in the convex case.
\begin{theorem}
  \label{thm:convex-epsilon-stationary-pt}
  Let  \cref{assumption:assumptions} and \cref{assumption:convex-assumptions} hold and choose $\varepsilon > 0$.
  Then, with either of the stepsize strategies stated in \cref{remark:stepsizes}, \cref{algorithm:CRPG} produces a point $\sequence{p}{k}$ such that there exists an $X_{\sequence{p}{k}} \in \partial f(\sequence{p}{k})$ with $\norm{X_{\sequence{p}{k}}} \le \varepsilon$ in at most
  \begin{equation}
    \label{eq:iteration-complexity-eps-stat}
    k
    =
    \left\lceil
    \frac{
      2(\lipgrad + \tilde \sigma_{\kmin, \kmax})
    }{
      \varepsilon
    }
    \sqrt{
      \frac{
        8
        \alphastepsize
      }{
        \tau \lipgrad
        \zetadelta
      }
    }
    \right\rceil
  \end{equation}
  iterations, where $\tilde \sigma_{\kmin, \kmax} \coloneq \sigmacurvature[\kmax]{\frac{\alphastepsize}{\lipgrad} \gradgupperbound + R}$, with $R \coloneq \diam(\cL_{\sequence{p}{0}})$, $\gradgupperbound$ as in \cref{remark:compact-level-sets}, $\alphastepsize$ is as in \cref{lemma:convex-stepsize-bounds}, and $\tau$ as in \cref{thm:convex-convergence-rate}.
\end{theorem}
\begin{proof}
    Since the stepsize strategies ensure that the sufficient decrease condition \cref{eq:convex-sufficient-decrease-condition} is satisfied, \cref{lemma:convex-stepsize-sufficient-decrease} gives
    \begin{equation}
        \label{eq:convex-sufficient-decrease-third-version}
        \sequence{\Delta}{k}
        \ge
        \sequence{\Delta}{k+1}
        +
        \frac{
          \zetadelta
        }{2 \sequence{\lambda}{k}}
        \dist^2(\sequence{p}{k}, \sequence{p}{k+1})
        ,
    \end{equation}
    with $\sequence{\Delta}{k} \coloneq f(\sequence{p}{k}) - f_{\text{opt}}$ for $k \ge 0$.
    The following is an adaptation of~\cite[Theorem~10.26]{Beck:2017:1}.
    By summing \cref{eq:convex-sufficient-decrease-third-version} over $n = k, k+1, \dots, 2k-1$, we obtain
    \begin{align*}
      \sum_{n=k}^{2k-1}
          \frac{
              \zetadelta
          }{2 \sequence{\lambda}{n}}
          \dist^2(\sequence{p}{n}, \sequence{p}{n+1})
      &
      \le
      \sequence{\Delta}{k}
      -
      \sequence{\Delta}{2k}
      \le
      \sequence{\Delta}{k}
      ,
    \end{align*}
    which implies
    \begin{equation*}
      k
      \min_{n=0,\dots, 2k-1}
          \frac{
              \zetadelta
          }{2 \sequence{\lambda}{n}}
          \dist^2(\sequence{p}{n}, \sequence{p}{n+1})
      \le
      \sequence{\Delta}{k}
      .
    \end{equation*}
    By \cref{eq:worst-case-convex-convergence-bound-3} we get
    \begin{equation*}
      \sequence{\Delta}{k}
      \le
      \frac{
          \sequence{\Delta}{0}
      }{
          1
          +
          k
          \tau
          \sequence{\Delta}{0}
      }
      <
      \frac{
          1
      }{
          k
          \tau
      }
      ,
    \end{equation*}
    and therefore
    \begin{equation*}
      \min_{n=0,\dots, 2k-1}
      \frac{
          \zetadelta
      }{2 \sequence{\lambda}{n}}
          \dist^2(\sequence{p}{n}, \sequence{p}{n+1})
      <
      \frac{
          1
      }{
          k^2
          \tau
      }
      ,
    \end{equation*}
    and finally
    \begin{equation}
      \label{eq:upper-bound-gradient-norm-k-squared}
      \min_{n=0,\dots, k}
      \frac{
          \zetadelta
      }{2 \sequence{\lambda}{n}}
          \dist^2(\sequence{p}{n}, \sequence{p}{n+1})
      <
      \frac{
          1
      }{
          \min\{
              (k/2)^2
              ,
              ((k+1)/2)^2
          \}
          \tau
      }
      \le
      \frac{
          4
      }{
          k^2
          \tau
      }
      ,
    \end{equation}
    where
    $
      \tau
      \coloneq
      \min\left\{
        \frac{\betastepsize}{
          2 \lipgrad \zetafirst{\frac{\alphastepsize}{\lipgrad} \gradgupperbound} R^2
        },
        \frac{1}{\sequence{\Delta}{0}}
      \right\}
    $.

    For the last step, we need~\cite[Lemma~4.8]{BergmannJasaJohnPfeffer:2025:1}.
    Note that such lemma assumes the stepsize to satisfy some nonconvex conditions, which are different from what we are assuming here.
    However, the result is still valid, since those conditions are there only to ensure sufficient decrease, which is already guaranteed by \cref{lemma:convex-stepsize-sufficient-decrease} here.
    In our case the lemma implies that for each $n$ there exists an $X_{\sequence{p}{n}}\in \partial f(\sequence{p}{n})$ with
    \begin{align*}
      \riemanniannorm{
          X_{\sequence{p}{n}}
      }^2
      &
      \le
      (\lipgrad + \tilde \sigma_{\kmin, \kmax})^2
      \dist^2(\sequence{p}{n}, \sequence{p}{n+1})
      \\
      &
      \le
      \frac{
          2
          \alphastepsize
          (\lipgrad + \tilde \sigma_{\kmin, \kmax})^2
      }{
          \zetadelta
          \lipgrad
      }
      \frac{
          \zetadelta
          }{2 \sequence{\lambda}{n}}
          \dist^2(\sequence{p}{n}, \sequence{p}{n+1})
      ,
    \end{align*}
    where $\tilde \sigma_{\kmin, \kmax} \coloneq \sigmacurvature[\kmax]{\frac{\alphastepsize}{\lipgrad} \gradgupperbound + R}$, $R \coloneq \diam(\cL_{\sequence{p}{0}})$, $\gradgupperbound$ as in \cref{remark:compact-level-sets}, and $\alphastepsize$ is as in \cref{lemma:convex-stepsize-bounds}.
    This means a sufficient condition for $\riemanniannorm{X_{\sequence{p}{n}}} \le \varepsilon$ is
    \begin{equation*}
        \frac{
            \zetadelta
        }{2 \sequence{\lambda}{n}}
            \dist^2(\sequence{p}{n}, \sequence{p}{n+1})
        \le
        \frac{
            \zetadelta
            \lipgrad \varepsilon^2
        }{
            2
            \alphastepsize
            (\lipgrad + \tilde \sigma_{\kmin, \kmax})^2
        },
    \end{equation*}
    which by \cref{eq:upper-bound-gradient-norm-k-squared}
    is the case for some $n \le k$ if \cref{eq:iteration-complexity-eps-stat} holds.
\end{proof}

\subsection{Strongly Convex Objective}

We now analyze the convergence behavior of the Riemannian proximal gradient method on Hadamard manifolds in the strongly convex case.
For this, we add the following assumption to \cref{assumption:assumptions}.
\begin{assumption}[Strong convexity] \hfill
	\label{assumption:strongly-convex-assumptions}
	\begin{itemize}
    \item $f = g + h$ is $\newtarget{def:strong-convex-constant}{\strongcvxconst}$-strongly convex on $\dom(g)$.
  \end{itemize}
\end{assumption}
The following theorem establishes a linear rate of convergence of the function values for \cref{algorithm:CRPG} on Hadamard manifolds in the strongly convex case.
This is a generalization of~\cite[Theorem~10.29]{Beck:2017:1}, though we note that the strategy to prove this result is fundamentally different from the Euclidean case.
We will address these differences in \Cref{section:proximal-gradient-inequalities}.
\begin{theorem}
  \label{thm:mu-strongly-convex-convergence-rate}
  Under the validity of \cref{assumption:assumptions}, \cref{assumption:convex-assumptions}, and \cref{assumption:strongly-convex-assumptions}, the sequence $\sequence(){p}{k}$ generated by \cref{algorithm:CRPG} with either of the stepsize strategies in \cref{remark:stepsizes} satisfies
  \begin{equation}
    \label{eq:mu-strongly-convex-recursion}
    \sequence{\Delta}{k+1}
    \le
    \left[
      1
      -
      \min\left\{
        \frac{\betastepsize \strongcvxconst}{4 \lipgrad \zetafirst{\frac{\alphastepsize}{\lipgrad} \gradgupperbound}}
          ,
          \frac{1}{2}
      \right\}
    \right]
    \sequence{\Delta}{k}
    ,
  \end{equation}
  for all $p^\ast \in \cS$ and $k \ge 0$, where $\gradgupperbound$ is as in \cref{remark:compact-level-sets}, and $\alphastepsize, \betastepsize$ are as in \cref{lemma:convex-stepsize-bounds}.
  In particular, for any $\varepsilon > 0$, $\sequence{\Delta}{k} \le \varepsilon$ is achieved for all
  \begin{equation}
    k
    \ge
    \left\lceil
    \max\left\{
        \frac{
            4 \lipgrad \zetafirst{\frac{\alphastepsize}{\lipgrad} \gradgupperbound}
        }{\betastepsize \strongcvxconst}
        ,
        2
    \right\}
    \log \frac{\sequence{\Delta}{0}}{\varepsilon}
    \right\rceil
    .
  \end{equation}
\end{theorem}
\begin{proof}
  The proof is an adaptation of~\cite[Proposition~22]{MartinezRubioPokuttaRoux:2025:1} carried out as follows.
  Since by \cref{assumption:convex-assumptions} $g$ is geodesically convex, \cref{eq:min-theta-convex-geodesic-convexity} holds.
  Because $f$ is $\strongcvxconst$-strongly geodesically convex, we can bound the distance term in \cref{eq:min-theta-convex-geodesic-convexity} by means of \cref{eq:strong-convexity} coupled with the fact that $0_{p^\ast} \in \partial f(p^\ast)$.
  This gives
  \begin{equation}
    \label{ineq:strongly-cvx-first}
    f(\sequence{p}{n+1})
    \le
    \min_{\theta \in [0,1]}\left\{
        f(\sequence{p}{n})
        - \theta \left[
            1
            -
            \frac{\zeta_{1, \kmin}(\sequence{D}{n}) \theta}{\sequence{\lambda}{n} \strongcvxconst}
        \right]
        \sequence{\Delta}{n}
    \right\}
    ,
  \end{equation}
  with $\sequence{D}{n} \coloneq \dist(\sequence{p}{n}, \gradientstep{\sequence{p}{n}})$.
  We get $
    \theta^\ast
    =
    \min\left\{
        1
        ,
        \frac{
            \sequence{\lambda}{n}\strongcvxconst
        }{
            2\zeta_{1, \kmin}(\sequence{D}{n})
        }
    \right\}
    $ by differentiating the argument of the minimum in \cref{ineq:strongly-cvx-first} with respect to $\theta$ and setting it to zero.

  If $\theta^\ast = 1$, then
  \begin{equation*}
    f(\sequence{p}{n+1})
    \le
    f(\sequence{p}{n})
    -
    \left[
        1
        -
        \frac{\zeta_{1, \kmin}(\sequence{D}{n})}{\sequence{\lambda}{n}\strongcvxconst}
    \right]
    \sequence{\Delta}{n}
    ,
  \end{equation*}
  and $\frac{\sequence{\lambda}{n}\strongcvxconst}{2\zeta_{1, \kmin}(\sequence{D}{n})} \ge 1$.
  Hence, $\frac{1}{2} \ge \frac{\zeta_{1, \kmin}(\sequence{D}{n})}{\sequence{\lambda}{n}\strongcvxconst}$, thus
  $
    f(\sequence{p}{n+1})
    \le
    f(\sequence{p}{n})
    -
    \frac{1}{2}\sequence{\Delta}{n}
    .
  $

  Otherwise,
  $
    f(\sequence{p}{n+1})
    \le
    f(\sequence{p}{n})
    -
    \frac{
        \sequence{\lambda}{n}\strongcvxconst
    }{
        4 \zeta_{1, \kmin}(\sequence{D}{n})
    }
    \sequence{\Delta}{n}
  .
  $
  Combining these cases, by \cref{eq:convex-stepsize-bounds} and the fact that $\sequence{D}{n} \le \frac{\alphastepsize}{\lipgrad} \gradgupperbound$ for all $n = 0, \ldots, k+1$ yields \cref{eq:mu-strongly-convex-recursion}.

  Notice that, by applying \cref{eq:mu-strongly-convex-recursion} recursively, one obtains
  \begin{equation*}
    \sequence{\Delta}{n}
    \le
    \left[
        1
        -
        \min\left\{
            \frac{\betastepsize \strongcvxconst}{4 \lipgrad \zetafirst{\frac{\alphastepsize}{\lipgrad} \gradgupperbound}}
            ,
            \frac{1}{2}
        \right\}
    \right]^{n}
    \sequence{\Delta}{0}
    .
  \end{equation*}
  A sufficient condition for $f(\sequence{p}{n}) - f(p^\ast) \le \varepsilon$ is given by
  \begin{equation*}
    n \log \left(
        1
        -
        \min\left\{
            \frac{\betastepsize \strongcvxconst}{4 \lipgrad \zetafirst{\frac{\alphastepsize}{\lipgrad} \gradgupperbound}}
            ,
            \frac{1}{2}
        \right\}
    \right)
    \le
    \log \frac{\varepsilon}{\sequence{\Delta}{0}}
    ,
  \end{equation*}
  and by $\log(1 - x) \le -x$ for any $x \le 1$, a sufficient condition for this inequality to hold is given by
  \begin{equation*}
    -
    n
    \min\left\{
            \frac{\betastepsize \strongcvxconst}{4 \lipgrad \zetafirst{\frac{\alphastepsize}{\lipgrad} \gradgupperbound}}
            ,
            \frac{1}{2}
        \right\}
    \le
    \log\frac{\varepsilon}{\sequence{\Delta}{0}}
    ,
  \end{equation*}
  which implies the claim.
\end{proof}
Importantly, this gives the following.
\begin{theorem}
  \label{thm:linear-convergence-iterates}
  Under \cref{assumption:assumptions} and \cref{assumption:strongly-convex-assumptions}, the sequence $\sequence(){p}{k}$ generated by \cref{algorithm:CRPG} converges linearly to the unique solution $p^\ast \in \cS$ of the optimization problem in \cref{eq:splitting}.
\end{theorem}
\begin{proof}
  By the strong convexity of $f$, $\cS = \{p^\ast\}$, see, \eg~\cite[Proposition~2.2.17]{Bacak:2014:3}, and the whole sequence $\sequence(){p}{k}$ generated by \cref{algorithm:CRPG} converges linearly to $p^\ast \in \cS$ as can be easily shown as follows.
  \begin{align*}
      \frac{\strongcvxconst}{2}
      \dist^2(\sequence{p}{k}, p^\ast)
      &
      \le
      f(\sequence{p}{k})
      -
      f(p^\ast)
      -
      \riemannian{
          0_{p^\ast}
      }{
          \logarithm{p^\ast}{\sequence{p}{k}}
      }
      \\
      &
      \le
      \left[
        1
        -
        \min\left\{
            \frac{\betastepsize \strongcvxconst}{4 \lipgrad \zetafirst{\frac{\alphastepsize}{\lipgrad} \gradgupperbound}}
            ,
            \frac{1}{2}
        \right\}
      \right]^{k}
      \sequence{\Delta}{0}
      ,
  \end{align*}
  where the last inequality follows from \cref{eq:mu-strongly-convex-recursion}.
\end{proof}

\section{Proximal Gradient Inequalities}
\label{section:proximal-gradient-inequalities}

This section provides some proximal gradient inequalities that generalize the fundamental prox-grad inequality presented in~\cite[Theorem~10.17]{Beck:2017:1} to Riemannian manifolds.
We discuss that, contrary to the Euclidean setting, these can be used to establish a convergence rate for the special cases of Riemannian proximal point and Riemannian gradient descent algorithms, but not for the Riemannian proximal gradient method on Hadamard manifolds.
We consider these results interesting in their own right, but moved the proofs to \cref{appendix}.

None of the following statements require the function $g$ to be geodesically convex.
Only the function $h$ is required to be geodesically convex.
However, assuming convexity of $g$ would not mitigate the shortcomings of the inequalities.
We also note that, in the following theorem, $\cM$ can be \emph{any} uniquely geodesic Riemannian manifold of bounded sectional curvature.
\begin{theorem}
  \label{thm:prox-grad-inequality}
  Let $\cM$ be a uniquely geodesic Riemannian manifold, and suppose \crefrange{assumption:assumptions:bounded-sectional-curvature}{assumption:assumptions:nonempty-optimal-set} of \cref{assumption:assumptions} hold.
  Let $\lambda > 0$, $p \in \cM$, and $q \in \interior(\dom(g))$ be such that
  \begin{equation}
      \label{eq:prox-grad-assumption}
      g\left(
          \Imap[auto]{\lambda}{q}
      \right)
      \le
      g(q)
      +
      \riemannian{
          \grad g(q)
      }
      {
          \logarithm{q}{\Imap[auto]{\lambda}{q}}
      }
      +
      \frac{1}{2 \lambda}
      \dist^2(
          q
          ,
          \Imap[auto]{\lambda}{q}
      )
      .
  \end{equation}
  Then
  \begin{equation}
      \label{eq:prox-grad-inequality-1}
      \begin{aligned}
          f(p)
          -
          f(\Imap[auto]{\lambda}{q})
          &
          \ge
          \glinear(p, q)
          +
          \frac{1}{2 \lambda}
          \dist^2(p, \Imap[auto]{\lambda}{q})
          -
          \frac{\zeta_{1, \kmin}(D_2)}{2 \lambda}
          \dist^2(p, q)
          \\
          &
          \quad
          +
          \frac{\zeta_{2, \kmax}(D_3) - 1}{2 \lambda}
          \dist^2(\gradientstep{q}, \Imap[auto]{\lambda}{q})
          +
          \frac{\zeta_{2, \kmax}(D_1) - 1}{2 \lambda}
          \dist^2(q, \gradientstep{q})
          ,
      \end{aligned}
  \end{equation}
  and
  \begin{equation}
      \label{eq:prox-grad-inequality-2}
      \begin{aligned}
          f(p)
          -
          f(\Imap[auto]{\lambda}{q})
          &
          \ge
          \glinear(p, q)
          +
          \frac{\zeta_{2, \kmax}(D_3)}{2 \lambda}
          \dist^2(p, \Imap[auto]{\lambda}{q})
          -
          \frac{1}{2 \lambda}
          \dist^2(p, q)
          \\
          &
          \quad
          +
          \frac{1 - \zeta_{1, \kmin}(D_2)}{2 \lambda}
          \dist^2(q, \gradientstep{q})
          +
          \frac{\zeta_{2, \kmax}(D_1) - 1}{2 \lambda}
          \dist^2(q, \Imap[auto]{\lambda}{q})
          ,
      \end{aligned}
  \end{equation}
  where $\glinear(p, q)$ is defined as
  \begin{equation}
      \label{eq:g-linearization}
      \newtarget{def:g-linearization}{
      \glinear(p, q)
      \coloneq
      g(p)
      -
      g(q)
      -
      \riemannian{
          \grad g(q)
      }
      {
          \logarithm{q}{p}
      }}
      ,
  \end{equation}
  and $D_1 = D(q, \gradientstep{q}, \Imap{\lambda}{q})$, $D_2 = D(q, \gradientstep{q}, p)$, and $D_3 = D(\Imap{\lambda}{q}, \gradientstep{q}, p)$, where $D(p, q, r)$ is the diameter of the uniquely geodesic triangle with vertices $p, q, r \in \cM$.
\end{theorem}
\begin{proof}
See Appendix~\ref{appendix:proof-prox-grad-inequality}.
\end{proof}
The next corollary is a special case of the fundamental prox-grad inequality on Hadamard manifolds.
\begin{corollary}
    \label{cor:fundamental-prox-grad-inequality}
    If \crefrange{assumption:assumptions:bounded-sectional-curvature}{assumption:assumptions:nonempty-optimal-set} of \cref{assumption:assumptions} hold, $\cM$ is a Hadamard manifold, and $g$ satisfies~\cref{eq:prox-grad-assumption}, then for any $p \in \cM$, $q \in \interior(\dom(g))$ the following inequality holds:
    \begin{equation}
        \label{eq:hadamard-prox-grad-inequality-1}
        \begin{aligned}
            f(p)
            -
            f(\Imap[auto]{\lambda}{q})
            &
            \ge
            \glinear(p, q)
            +
            \frac{1}{2 \lambda}
            \dist^2(p, \Imap[auto]{\lambda}{q})
            -
            \frac{\zeta_{1, \kmin}(\dist(q, \gradientstep{q})) + 1}{4 \lambda}
            \dist^2(p, q)
            \\
            &
            \quad
            -
            \frac{\zeta_{1, \kmin}(\dist(p, q)) - 1}{4 \lambda}
            \dist^2(q, \gradientstep{q})
            .
        \end{aligned}
    \end{equation}
\end{corollary}
\begin{proof}
See Appendix~\ref{appendix:proof-fundamental-prox-grad-inequality}.
\end{proof}

The following corollary provides a second sufficient decrease condition compared to~\cite[Lemma~4.4]{BergmannJasaJohnPfeffer:2025:1}, which is a generalization of~\cite[Corollary~10.18]{Beck:2017:1}.
\begin{corollary}
    \label{cor:sufficient-decrease-second-version}
    If \cref{assumption:assumptions} holds and $\cM$ is a Hadamard manifold, then for any $p \in \interior(\dom(g))$ such that%
    \begin{equation}
        \label{eq:sufficient-decrease-second-version}
        g(\Imap[auto]{\lambda}{p})
        \le
        g(p)
        +
        \riemannian{
            \grad g(p)
        }
        {
            \logarithm{p}{\Imap[auto]{\lambda}{p}}
        }
        +
        \frac{1}{2 \lambda}
        \dist^2(p, \Imap[auto]{\lambda}{p})
        ,
    \end{equation}
    the following inequality holds:
    \begin{equation}
        \label{eq:sufficient-decrease-second-version-2}
        f(p)
        -
        f(\Imap[auto]{\lambda}{p})
        \ge
        \frac{1}{2 \lambda}
        \dist^2(p, \Imap[auto]{\lambda}{p})
        .
    \end{equation}
\end{corollary}
\begin{proof}
    The claim follows by setting $q = p$ in \cref{eq:hadamard-prox-grad-inequality-1}.
\end{proof}

In the following, we provide a discussion on the function values' convergence rates of the Riemannian proximal point (RPP) algorithm, \cf~\cite{FerreiraOliveira:2002:1}, and the Riemannian gradient descent (RGD) algorithm, \cf~\cite{ZhangSra:2016:1}, on Hadamard manifolds following an approach exploiting the fundamental proximal gradient inequalities.
Both algorithms can be seen as special cases of \cref{algorithm:CRPG} with $g = 0$ and $h \neq 0$ for RPP, and $g \neq 0$ and $h = 0$ for RGD.
We also outline how this approach provides known convergence rates for RPP and RGD, but not for \cref{algorithm:CRPG}, in contrast to the Euclidean case in~\cite[Theorem~10.21]{Beck:2017:1}.
\begin{theorem}
  \label{thm:rpp-rgd-convergence-rates}
  Let $\cM$ be a Hadamard manifold, and let $R \coloneq \dist(p^\ast, \sequence{p}{0})$, with $p^\ast \in \cS$ a solution of the optimization problem in \cref{eq:splitting}.
  Denote with $\sequence{\Delta}{k}$ the optimal function gap at iteration $k \in \bbN$.
  Under the assumptions of \cref{thm:convex-convergence-rate}, the following hold.
  \begin{itemize}
    \item If $g = 0$,
      \begin{equation*}
          \sequence{\Delta}{k}
          \le
          \frac{R \lipgrad}{2 \betastepsize k}
          ,
      \end{equation*}
      where $\betastepsize$ is as in \cref{lemma:convex-stepsize-bounds}.

    \item If $h = 0$ and the stepsize $\lambda$ is constant throughout the iterations,
      \begin{equation*}
          \sequence{\Delta}{k}
          \le
          \frac{
              R^2
              +
              2 \lambda \zeta_{1, \kmin}(R) \sequence{\Delta}{0}
          }{
              2 \lambda (
                  \zeta_{1, \kmin}(R)
                  +
                  k
              )
          }
          .
      \end{equation*}
  \end{itemize}
\end{theorem}
\begin{proof}
  See \cref{appendix:proof-convergence-rates}.
\end{proof}

In general, $g, h \neq 0$, and the term involving $\dist^2(q, \gradientstep{q})$ in \cref{eq:hadamard-prox-grad-inequality-1} cannot be canceled out using the sufficient decrease condition.
Therefore, the fundamental proximal gradient inequality is not sufficient to directly provide a convergence rate of the function values for \cref{algorithm:CRPG}, as no telescoping can be easily achieved following this path.

Furthermore, in the Euclidean setting, the proximal gradient inequality is exploited in~\cite[Theorem~10.23]{Beck:2017:1} to establish the Fejér monotonicity of the sequence $\sequence(){p}{n}$ in the convex case, which is then used in~\cite[Theorem~10.24]{Beck:2017:1} to prove convergence of such sequence to an optimal solution of the problem \cref{eq:splitting}.
The same cannot be achieved here.
The reason is that none of the fundamental proximal gradient inequalities imply the Fejér monotonicity of the sequence $\sequence(){p}{n}$ in the geodesically convex case.
Indeed, taking \cref{eq:prox-grad-convergence-failure} and dropping the positive terms gives
\begin{align*}
    \dist^2(\sequence{p}{n+1}, p^\ast)
    &
    \le
    \frac{
        \zetafirst{\dist(\sequence{p}{n}, z_{\sequence{p}{n}})}
        +
        1
    }{2}
    \dist^2(\sequence{p}{n}, p^\ast)
    \\
    &
    \quad
    +
    \frac{
        \zetafirst{\dist(p^\ast, \sequence{p}{n})}
        -
        1
    }{2}
    \dist^2(\sequence{p}{n}, z_{\sequence{p}{n}})
    ,
\end{align*}
which does not imply any Fejér monotonicity.

\section{Numerical Examples}
\label{section:numerics}

In this section, we present numerical experiments to evaluate the performance of the Riemannian Proximal Gradient Method (CRPG) in \cref{algorithm:CRPG} in both the convex and strongly convex case.

CRPG is implemented in \julia using \manoptjl{}~\cite{Bergmann:2022:1} so that it can be used with manifolds from \manifoldsjl{}~\cite{AxenBaranBergmannRzecki:2023}.
Unless otherwise stated, we cap CRPG at $5000$ iterations and set the default stopping criterion to a tolerance of $10^{-7}$ on the norm of the gradient mapping.
The warm start factor $\theta$ in \cref{line:warm-start} of \cref{algorithm:backtracking} is set to a default of $1$.
All numerical experiments were performed on a MacBook~Pro~M1, 16~GB~RAM running macOS~15.5.
These experiments are available as notebooks in the package \manoptexamplesjl~\cite{BergmannJasa:2025:17277311}.

\subsection{Convex Example on Symmetric Positive Definite Matrices}
\label{subsection:convex-example-symmetric-positive-definite-matrices}

In our first example, we consider the space of $n$-by-$n$ symmetric positive definite matrices $\cP(n)$, which, when equipped with the affine invariant metric, is complete and hence a Hadamard manifold, see, \eg,~\cite[Proposition~3.1]{DolcettiPertici:2018:1}. Its sectional curvature is nonpositive and bounded from below by $\kmin = -\frac{1}{2}$; see, \eg,~\cite[Appendix~I.1, Proposition~I.1]{CriscitielloBoumal:2020:1}.

Let $\cM = \cP(n)$ with inner product $\riemannian{X_p}{Y_p} = \trace(p^{-1} X_p p^{-1} Y_p)$ for $X_p, Y_p \in T_p \cM$.
For $\tau > 0$ and $\bar q \in \cM$, we solve the optimization problem in \cref{eq:splitting} with
\begin{equation*}
    g(p) \coloneq (\log(\det(p)))^4
    \quad
    \text{and}
    \quad
    h(p) \coloneq \tau \dist(p, \bar q)
    .
\end{equation*}
By~\cite[Example~6.1]{BergmannFerreiraSantosSouza:2024}, $g$ is geodesically convex on $\cM$ with gradient and Hessian
\begin{equation*}
    \grad g(p) = 4 (\log(\det(p)))^3 p
    \quad
    \text{and}
    \quad
    \Hess g(p) [X_p] = 12 (\log(\det(p)))^2 \trace(p^{-1} X_p) p
    ,
\end{equation*}
respectively.
Hence, $\riemannian{\Hess g(p) [X_p]}{X_p} = 12 (\log(\det(p)))^2 \trace(p^{-1} X_p)^2$.
One can check that the smallest and largest eigenvalues of $\Hess g(p)$ are $\lambda_{\text{min}} = 0$ and $\lambda_{\text{max}} = 12 n (\log(\det(p)))^2 > 0$.

Finally, the proximal mapping of $h$ is
$
    \prox_{\lambda h}(p)
    =
    \geodesic<l>{p}{\bar q}(
        \min\{
            \lambda \tau
            ,
            \dist(p, \bar q)
        \}
    )
    ,
$
where $\geodesic<l>{p}{\bar q}$ is the unit-speed geodesic that starts from $p$ in the direction of $\bar q$.
See, \eg,~\cite[Proposition~1]{WeinmannDemaretStorath:2014:1}.
\begin{table}[tbp]
    \centering
    \begin{adjustbox}{max width=\textwidth}
    \begin{filecontents*}{data/spd/CRPG-Convex-SPD-Comparisons.csv}
3,367,0.004133187500000001,0.1859296378876703,33,0.003139625,0.18592963788766845
6,1944,0.0419539375,0.2707795510146935,56,0.0092925,0.2707795510146913
10,8640,0.268461333,0.37127350509491613,91,0.021167104000000003,0.37127350509491425
15,15535,0.6224665,0.44962526975264777,121,0.037344958000000004,0.4496252697526453
\end{filecontents*}
\pgfplotstabletypeset[col sep = comma,
    columns={0,1,2,4,5,6}, % Skip column 3
    every head row/.style = {before row = \toprule
        \multicolumn{1}{c}{} & \multicolumn{2}{c}{Constant} & \multicolumn{2}{c}{Backtracked}\\
        \cmidrule(lr){2-3}\cmidrule(lr){4-5},
    after row = \midrule},
    every last row/.style = {after row = \midrule},
    display columns/0/.style = {
        column name = Dimension,
        string type,
        column type = {S[table-number-alignment = right, table-format = 5, table-alignment-mode = format]},
    },
    display columns/1/.style = {
        column name = Iter.,
        string type,
        column type = {S[table-number-alignment = right, table-format = 5, table-alignment-mode = format]},
    },
    display columns/2/.style = {
        column name = Time (sec.),
        string type,
        column type = {S[table-number-alignment = right, table-auto-round = true, scientific-notation = engineering, exponent-mode = scientific, table-format = 1.2e1]},
    },
    display columns/3/.style = {
        column name = Iter.,
        string type,
        column type = {S[table-number-alignment = right, table-format = 5, table-alignment-mode = format]},
    },
    display columns/4/.style = {
        column name = Time (sec.),
        string type,
        column type = {S[table-number-alignment = right, table-auto-round = true, scientific-notation = engineering, exponent-mode = scientific, table-format = 1.2e1]},
    },
    display columns/5/.style = {
        column name = Objective,
        string type,
        column type = {S[table-number-alignment = right, table-auto-round = true, table-format = 1.6, table-alignment-mode = format]},
    },
    multicolumn names,
    % mark by best iterations
    every row 0 column 3/.style={postproc cell content/.style={@cell content={\color{table-highlight-best}##1}}},
    every row 1 column 3/.style={postproc cell content/.style={@cell content={\color{table-highlight-best}##1}}},
    every row 2 column 3/.style={postproc cell content/.style={@cell content={\color{table-highlight-best}##1}}},
    every row 3 column 3/.style={postproc cell content/.style={@cell content={\color{table-highlight-best}##1}}},
    % mark by best times
    every row 0 column 4/.style={postproc cell content/.style={@cell content={\color{table-highlight-best}##1}}},
    every row 1 column 4/.style={postproc cell content/.style={@cell content={\color{table-highlight-best}##1}}},
    every row 2 column 4/.style={postproc cell content/.style={@cell content={\color{table-highlight-best}##1}}},
    every row 3 column 4/.style={postproc cell content/.style={@cell content={\color{table-highlight-best}##1}}},
    ]
    {data/spd/CRPG-Convex-SPD-Comparisons.csv}
  \end{adjustbox}
  \vspace{1em}
	\caption{
        Comparison between CRPG with a constant and backtracked stepsize on $\cP(n)$ with varying dimensions.
        Both methods converged to solutions with objective values that differ by $10^{-11}$ at most.
        Example from \cref{subsection:convex-example-symmetric-positive-definite-matrices}.
    }
	\label{table:CRPG-SPD-comparisons}
\end{table}

\begin{figure}
    % Read the CSV file
    \pgfplotstableread[col sep=comma]{data/spd/CRPG-Convex-SPD-Constant-2.csv}\constantstepsize
    \pgfplotstableread[col sep=comma]{data/spd/CRPG-Convex-SPD-Backtracking-2.csv}\backtrackedstepsize

    \centering
    \begin{tikzpicture}
        \begin{semilogyaxis}[
            height=0.5\textwidth,
            width=0.8\textwidth,
            xlabel={iterations},
            ylabel={function value error},
            ylabel near ticks,
            xlabel near ticks,
            ymajorgrids,
            %xmin=.5, xmax=60,
            % ymin=1e-10,
            label style={font=\scriptsize},
            tick label style={font=\scriptsize},
            legend style={
                font=\scriptsize,
                nodes={scale=0.66},
                at={(0.975,0.6)},
                anchor=east,
            },
            legend cell align=left,
        ]
            \addplot[color=TolBrightBlue,thick] table[x=Iteration, y=Diff] {\constantstepsize};
            \addplot[color=TolBrightCyan,thick] table[x=Iteration, y=Diff] {\backtrackedstepsize};
            \addplot[color=TolBrightYellow,line width=1.5pt,domain=1:370,samples=130,dashed] {1/x};
            \addplot[color=TolBrightGreen,line width=1.5pt,domain=1:51,samples=130,dashed] {1/2^x};
            \legend{CRPG constant stepsize, CRPG backtracking, $\cO(1/k)$, $\cO(1/2^k)$};
        \end{semilogyaxis}
    \end{tikzpicture}
    \caption{Error of the cost function $f(\sequence{p}{k}) - f_{\mathrm{opt}}$ on $\cP(2)$ for CRPG with a constant and backtracked stepsizes.
    Example from \cref{subsection:convex-example-symmetric-positive-definite-matrices}.}
    \label{fig:CRPG-SPD-cost-error}
\end{figure}

\paragraph{\textbf{Implementation Details}}
\footnote{The code is available at \url{https://juliamanifolds.github.io/ManoptExamples.jl/stable/examples/CRPG-Convex-SPD/}}
We estimated the Lipschitz constant $\lipgrad$ of the gradient of $g$ over a geodesic ball of radius $2 \dist(\sequence{p}{0}, \bar q)$ around a randomly chosen starting point $\sequence{p}{0}$.
The constant step size is $\lambda = \frac{1}{\lipgrad}$ in \cref{algorithm:CRPG}.
For the backtracking procedure, we use an initial guess $s = 1$, while the contraction factor is set to $\eta = 0.9$ and the warm start factor to $\theta = 2$.
The iteration cap was set to $20 000$.
The weight in the objective function $f$ was set to $\tau = \frac{1}{2}$.

\paragraph{\textbf{Discussion}}
The results of the numerical experiment are shown in \cref{table:CRPG-SPD-comparisons}.
The first column displays the manifold dimension $\frac{n(n+1)}{2}$, where $n \in \{2, 3, 4, 5\}$.
As can be seen, the runs with a constant stepsize converge in more iterations than the one with a backtracked stepsize, are up to an order of magnitude slower than the latter, as can be seen in the last three rows of the table.
\cref{fig:CRPG-SPD-cost-error} shows the error of the cost function $f(\sequence{p}{k}) - f_{\mathrm{opt}}$, for both stepsize strategies.
The error is plotted against the number of iterations, which shows that the backtracked stepsize converges faster than the constant stepsize.
The dotted lines show the rates $\cO(1/k)$ and $\cO(2^{-k})$.
This is consistent with \cref{thm:convex-convergence-rate}.

\subsection{Strongly Convex Example on the Sphere}
\label{subsection:convex-example-sphere}

The following example illustrates the convergence rate obtained in \cref{thm:mu-strongly-convex-convergence-rate} for our method on manifolds with positive curvature.
In contrast, the analysis of the proximal gradient method presented in~\cite{FengHuangSongYingZeng:2021} is limited to manifolds with non‑positive curvature and therefore provides \emph{no} convergence guarantee on the sphere.

We solve the problem of computing the Fréchet mean of $N=1000$ points that lie in a geodesically convex ball on the sphere, while simultaneously encouraging proximity to a reference point $\bar q$ that belongs to the same ball.
Thus we consider \cref{eq:splitting} with
\begin{equation}\label{eq:sphere-example}
    g(p)=\frac{1}{2N}\sum_{j=1}^{N}\dist^{2}(p,q_{j}),
    \qquad
    h(p)=\tau\,\dist(p,\bar q),
\end{equation}
where $\bar q, q_{1},\dots,q_{N}\in\cB\bigl(q_0,r\bigr)$ are randomly generated points in a geodesically convex ball of radius $r<\frac{\pi}{4}$ around some fixed center $q_{0}\in\mathbb S^{n}$, and the weighting parameter is chosen as $\tau=\frac{1}{10}$.

On the unit sphere $\mathbb S^{n}$, the distance function $p\mapsto\dist(p,q)$ and the squared distance ${p\mapsto\dist^{2}(p,q)}$ are geodesically convex on every geodesic ball of radius strictly smaller than $\frac \pi 2$.
See, \eg,~\cite[Proposition~4.2]{LiYao:2012:1} and~\cite{FerreiraIusemNemeth:2014:1}.
Furthermore, the function $p\mapsto \frac 1 2 \dist^{2}(p,q)$ is $\zeta_{2, 1}(2 r)$-strongly geodesically convex by~\cite[Corollary~2.1]{AlimisisBecigneulLucchiOrvieto:2020:1}.
And the proximal mapping of $h$ is given by
${
    \prox_{\lambda h}(p)
    =
    \geodesic<l>{p}{\bar q}(
        \min\{
            \lambda
            ,
            \dist(p, \bar q)
        \}
    )
    ,
}$
where $\geodesic<l>{p}{\bar q}$ is the unit-speed geodesic that starts from $p$ in the direction of $\bar q$.
See, \eg,~\cite[Proposition~1]{WeinmannDemaretStorath:2014:1}.

\begin{table}[tbp]
    \centering
    \begin{adjustbox}{max width=\textwidth}
    \begin{filecontents*}{data/sphere/results.csv}
100,57.5,0.0231786041,0.07756373254326748,21.1,0.013062735249999999,0.07756373252692303
500,54.7,0.10248221895000001,0.06326206038020402,19.9,0.0538948498,0.06326206036562201
1000,61.8,0.1961269187,0.07549692413831331,22.4,0.10993444370000001,0.07549692412032058
5000,43.9,0.5346263706000001,0.06181206825630966,16.4,0.33553416454999996,0.061812068245510624
\end{filecontents*}
\pgfplotstabletypeset[col sep = comma,
    columns={0,1,2,4,5,6}, % Skip column 3
    every head row/.style = {before row = \toprule
        \multicolumn{1}{c}{} & \multicolumn{2}{c}{Constant} & \multicolumn{2}{c}{Backtracked}\\
        \cmidrule(lr){2-3}\cmidrule(lr){4-5},
    after row = \midrule},
    every last row/.style = {after row = \midrule},
    display columns/0/.style = {
        column name = Dimension,
        string type,
        column type = {S[table-number-alignment = right, table-format = 5, table-alignment-mode = format]},
    },
    display columns/1/.style = {
        column name = Iter.,
        string type,
        column type = {S[table-number-alignment = right, table-format = 5, table-alignment-mode = format]},
    },
    display columns/2/.style = {
        column name = Time (sec.),
        string type,
        column type = {S[table-number-alignment = right, table-auto-round = true, scientific-notation = engineering, exponent-mode = scientific, table-format = 1.2e1]},
    },
    display columns/3/.style = {
        column name = Iter.,
        string type,
        column type = {S[table-number-alignment = right, table-format = 5, table-alignment-mode = format]},
    },
    display columns/4/.style = {
        column name = Time (sec.),
        string type,
        column type = {S[table-number-alignment = right, table-auto-round = true, scientific-notation = engineering, exponent-mode = scientific, table-format = 1.2e1]},
    },
    display columns/5/.style = {
        column name = Objective,
        string type,
        column type = {S[table-number-alignment = right, table-auto-round = true, table-format = 1.6, table-alignment-mode = format]},
    },
    multicolumn names,
    % mark by best iterations
    every row 0 column 3/.style={postproc cell content/.style={@cell content={\color{table-highlight-best}##1}}},
    every row 1 column 3/.style={postproc cell content/.style={@cell content={\color{table-highlight-best}##1}}},
    every row 2 column 3/.style={postproc cell content/.style={@cell content={\color{table-highlight-best}##1}}},
    every row 3 column 3/.style={postproc cell content/.style={@cell content={\color{table-highlight-best}##1}}},
    % mark by best times
    every row 0 column 4/.style={postproc cell content/.style={@cell content={\color{table-highlight-best}##1}}},
    every row 1 column 4/.style={postproc cell content/.style={@cell content={\color{table-highlight-best}##1}}},
    every row 2 column 4/.style={postproc cell content/.style={@cell content={\color{table-highlight-best}##1}}},
    every row 3 column 4/.style={postproc cell content/.style={@cell content={\color{table-highlight-best}##1}}},
    ]
    {data/sphere/results.csv}
  \end{adjustbox}
  \vspace{1em}
	\caption{
        Comparison between CRPG with a constant and backtracked stepsize on $\mathbb S^n$ with varying dimensions. The results are averaged over 10 runs.
        Both methods converged to solutions with objective values that differ by $10^{-11}$ at most.
        Example from \cref{subsection:convex-example-sphere} with $r = \frac{\pi}{6}$, $s = 10 \lambda$ and $\delta =1.0$.
    }
	\label{table:Sphere-example}
\end{table}

\begin{figure}[tbp]
    \pgfplotstableread[col sep=comma]{data/sphere/results-02pi-radius-const.csv}\constantstepsizea
    \pgfplotstableread[col sep=comma]{data/sphere/results-02pi-radius-bt.csv}\backtrackedstepsizea
    \pgfplotstableread[col sep=comma]{data/sphere/results-0125pi-radius-const.csv}\constantstepsizeb
    \pgfplotstableread[col sep=comma]{data/sphere/results-0125pi-radius-bt.csv}\backtrackedstepsizeb
        \pgfplotstableread[col sep = comma]{data/constrained-mean/CRPG-CnBallConstrMean-2-200-times.csv}\datatabletimes
    \pgfplotstableread[col sep = comma]{data/constrained-mean/CRPG-CnBallConstrMean-2-200-iterations.csv}\datatableiterations
    \begin{subfigure}{.5\textwidth}
        \centering
        \begin{tikzpicture}
            \begin{semilogyaxis}[
                width=0.99\textwidth,
                xlabel={iterations},
                ylabel={function value error},
                ylabel near ticks,
                xlabel near ticks,
                ymajorgrids,
                %xmin=.5, xmax=60,
                % ymin=1e-10,
                label style={font=\scriptsize},
                tick label style={font=\scriptsize},
                legend style={
                    font=\small,
                    nodes={scale=0.5},
                    at={(0.52,1.00)},
                    anchor=south,
                    row sep=-0.3ex,
                    column sep=0.7ex,
                    draw=none
                },
                legend columns=2,
                legend cell align=left,
            ]
                \addplot[color=TolBrightBlue,line width = 1.0pt,
                        mark= Mercedes star, mark options={line width = 0.5pt},
                        mark repeat = 3, mark phase = 1]
                        table[x=Iter, y=delta0.113] {\constantstepsizea};
                \addplot[color=TolBrightGreen, line width = 1.0, domain=1:150, samples=130, dashed, opacity=1.0] {(0.987)^x};
                \addplot[color=TolBrightCyan,line width = 1.0pt,
                        mark = triangle, mark options={line width = 0.5pt, scale = 0.8},
                        mark repeat = 5, mark phase = 1]
                        table[x=Iter, y=delta0.1] {\backtrackedstepsizea};
                \addplot[color=TolBrightRed,line width=1.0pt,domain=1:150,samples=130,dashed, opacity=0.5] {(0.9)^x};
                \addplot[color=TolBrightCyan,line width = 1.0pt,
                        mark =Mercedes star flipped, mark options={line width = 0.5pt},
                        mark repeat = 2, mark phase = 2]
                        table[x=Iter, y=delta1.0] {\backtrackedstepsizea};
                \addplot[color=TolBrightRed,line width=1.0pt,domain=1:60,samples=130,dashed, opacity=0.7] {(0.7)^x};
                \addplot[color=TolBrightCyan,line width = 1.0pt,
                        mark = o, mark options={line width = 0.5pt, scale = 0.8},
                        mark repeat = 2, mark phase = 1]
                        table[x=Iter, y=delta100.0] {\backtrackedstepsizea};
                \addplot[color=TolBrightRed,line width=1.0pt,domain=1:31,samples=130,dashed, opacity=0.9] {(0.5)^x};
                \legend{{Constant, $\bar\delta\!=\!0.113$},
                    $\cO(0.987^k)$ (Thm. \ref{thm:mu-strongly-convex-convergence-rate}),
                    {Backtracked, $\delta\!=\!0.10$},%
                    $\cO(0.9^k)$,
                    {Backtracked, $\delta\!=\!1.0$},%
                    $\cO(0.7^k)$,
                    {Backtracked, $\delta\!=\!100.0$},%
                    $\cO(0.5^k)$};
            \end{semilogyaxis}
        \end{tikzpicture}
    \caption{$ \bar q, q_1, ..., q_N \in\cB(q_0, r)$, where $r = \frac{\pi}{5}$.
    %\\ and $\lambda = 0.125$, $s = 0.126$
    }
    \label{subfig:sphere-radius-02pi}
    \end{subfigure}%
    \begin{subfigure}{.5\textwidth}
        \centering
        \begin{tikzpicture}
            \begin{semilogyaxis}[
                width=0.99\textwidth,
                xlabel={iterations},
                ylabel={function value error},
                ylabel near ticks,
                xlabel near ticks,
                ymajorgrids,
                %xmin=.5, xmax=60,
                % ymin=1e-10,
                label style={font=\scriptsize},
                tick label style={font=\scriptsize},
                legend style={
                    font=\small,
                    nodes={scale=0.5},
                    at={(0.52,1.0)},
                    anchor=south,
                    row sep=-0.3ex,
                    column sep=0.7ex,
                    draw=none
                },
                legend columns=2,
                legend cell align=left,
            ]
                \addplot[color=TolBrightBlue,line width = 1.0pt,
                        mark= Mercedes star, mark options={line width = 0.5pt},
                        mark repeat = 3, mark phase = 1]
                        table[x=Iter, y=delta0.184] {\constantstepsizeb};
                \addplot[color=TolBrightGreen, line width = 1.0pt,domain=1:155, samples=130,dashed, opacity=1.0] {(0.963)^x};
                \addplot[color=TolBrightCyan,line width = 1.0pt,
                        mark = triangle, mark options={line width = 0.5pt, scale = 0.8},
                        mark repeat = 5, mark phase = 1]
                        table[x=Iter, y=delta0.1] {\backtrackedstepsizeb};
                \addplot[color=TolBrightRed,line width=1.0pt,domain=1:155,samples=100,dashed, opacity=0.5] {(0.9)^x};
                \addplot[color=TolBrightCyan,line width = 1.0pt,
                        mark =Mercedes star flipped, mark options={line width = 0.5pt},
                        mark repeat = 1, mark phase = 1]
                        table[x=Iter, y=delta1.0] {\backtrackedstepsizeb};
                \addplot[color=TolBrightRed,line width=1.0pt,domain=1:60,samples=130,dashed, opacity=0.7] {(0.7)^x};
                \addplot[color=TolBrightCyan,line width = 1.0pt,
                        mark = o, mark options={line width = 0.5pt, scale = 0.8},
                        mark repeat = 1, mark phase = 1]
                        table[x=Iter, y=delta100.0] {\backtrackedstepsizeb};
                \addplot[color=TolBrightRed,line width=1.0pt,domain=1:31,samples=130,dashed, opacity=0.9] {(0.5)^x};
                \legend{{Constant, $\bar \delta\!=\!0.184$},
                    $\cO(0.963^k)$ (Thm. \ref{thm:mu-strongly-convex-convergence-rate}),
                    {Backtracked, $\delta\!=\!0.1$},%
                    $\cO(0.9^k)$,
                    {Backtracked, $\delta\!=\!1.0$},%
                    $\cO(0.7^k)$,
                    {Backtracked, $\delta\!=\!100.0$},%
                    $\cO(0.5^k)$};
            \end{semilogyaxis}
        \end{tikzpicture}
        \caption{$ \bar q, q_1, ..., q_N \in\cB(q_0, r)$, where $r = \frac{\pi}{8}$.
        %\\ and $\lambda = 0.191$, $s = 0.201$
        }
        \label{subfig:sphere-radius-0125pi}
    \end{subfigure}%
    \caption{Error of the cost function $f(\sequence{p}{k}) - f_{\mathrm{opt}}$ on $\mathbb S^n$ with $n = 500$ for CRPG with optimized constant stepsize and backtracked stepsizes for different values of $\delta$.
    Example from \cref{subsection:convex-example-sphere}.}
\label{fig:perturbed-mean}
\end{figure}

\paragraph{\textbf{Implementation Details}}
\footnote{The code is available at \url{https://juliamanifolds.github.io/ManoptExamples.jl/stable/examples/CRPG-Sphere-Example/}}
First, the anchor point $q_0 \in \mathbb S^n$ was drawn uniformly at random.
Then, $N+2$ tangent vectors at $q_0$ were sampled from a normal distribution with standard deviation $\sigma = 1$.
Each vector was then scaled down to a length strictly less than $r$ and mapped onto $\mathbb S^n$ using $\exponential{q_0}$.
The first $N$ resulting points $\{q_1, \ldots, q_N\}$ were used as data points for the function $g$. The remaining two points were used as the reference point $\bar{q}$ for $h$, and the initial point $p_0$ of the optimization, respectively.
Stepsizes were chosen according to the rules described in \cref{remark:stepsizes} and \cref{algorithm:backtracking}. For the constant stepsize, we chose $\bar\delta$ such that $\lambda = \min\{\lambdadelta, \frac{\zetadelta}{\lipgrad} \}$ is maximized.
In the backtracked case, we chose the initial stepsize $s = 10 \lambda$ with varying $\delta$ and parameters $\eta = 0.5$ and $\theta = 2.0$.
\paragraph{\textbf{Discussion}}
\cref{fig:perturbed-mean} shows the convergence behavior of the error in the cost function for the optimized stepsize, as well as for the backtracked one with different initial stepsizes $s$, depending on different values of $\delta$, on subsets of radii $r \in \left\{\frac{\pi}{5}, \frac{\pi}{8}\right\}$.
A linear convergence rate is observed for all step size choices, which is consistent with \cref{thm:mu-strongly-convex-convergence-rate}. 
The rate depends on the problem setting. 
The bounds of $h$ and the gradient of $g$ that are used to compute $\lambdadelta$ depend linearly on the radius $r$ of the set.
Specifically, $\hlowerbound = 0$,  $\hupperbound = 2 \tau r$, and $\gradgupperbound = 2 r$.
We observe that, in the backtracked case, the larger the value of $\delta$, the better the performance of the method.

\subsection{Sparse Mean on $\cH^n$}
\label{subsection:numerics-sparse-approximation}

In the next experiment, we want to compute the mean of $N$ points on the hyperbolic space $\cH^n$ with additional sparsity constraints.
That is, we solve the optimization problem in \cref{eq:splitting} on $\cM = \cH^n$, the hyperbolic space equipped with the Minkowski metric of signature $(+, \ldots, +, -)$.
Then $\cH^n$ is a Hadamard manifold with constant sectional curvature $\kmin = -1$.

For the computation of the sparse mean, we set
\begin{equation*}
g(p) = \frac{1}{2N}\sum_{i=1}^N \dist^2(p, q_i)
\end{equation*} for $p \in \cM$ and $q_i \in \cM$ for $i = 1, \ldots, N$, where $N = 1000$ and $q_i$ are randomly sampled points on $\cM$.
Furthermore, we choose $h(p) = \Vert p \Vert_1$, which is the $\ell_1$-norm on the hyperbolic space.

\paragraph{Soft Thresholding on $\cH^n$}
\label{subsubsection:soft-thresholding-Hn}

We now set out to derive a method for computing the proximal operator of the $\ell_1$-norm on the hyperbolic space $\cH^n$ that is easy to implement, while arguing that the function is geodesically convex.

We represent the hyperbolic space $\cH^n$ as a subset of $\mathbb R^{n+1}$ using the Minkowski pseudo inner product $\langle{x},{y}\rangle_M = x^\top J y$ with $J = \diag(1, \ldots, 1, -1)$ as
\begin{equation*}
    \cH^n
    =
    \setDef{x \in \mathbb R^{n+1}}{\langle x, x\rangle_M = -1,\, x_{n+1} >0}
    ,
\end{equation*}
with the metric $g_M = \langle \cdot, \cdot \rangle_M$.
Then the tangent space at $x \in \cH^n$ and the projection of a vector $z \in \mathbb R^{n+1}$ onto it are given by
$
    \tangentSpace{x}[\cH^n]
    =
    \setDef{v \in \mathbb R^{n+1}}{\langle x, v\rangle_M = 0},
$ and
$
    \proj{x}(z)
    =
    z
    +
    \langle x, z\rangle_M
    \,
    x
    ,
$
see, \eg,~\cite[Chapter~7.6]{Boumal:2023:1}.
The geodesic distance is defined as
$
    \dist(x, y)
    =
    \text{arccosh}\left(- \langle x, y\rangle_M\right)
    ,
$
and the exponential and logarithm maps are
\begin{align*}
    \exponential{x}(v)
    &= \cosh{
        \left(\sqrt{\langle v,v\rangle_M}\right)
        }
        x
    +
    \sinh{\left(\sqrt{\langle v,v\rangle_M}\right)}
        \frac v {\sqrt{\langle v,v\rangle_M}},
    \\
    \logarithm{x}(y)
    &=
    \frac{
        \text{arcosh}\left(-\langle x, y\rangle_M \right)
    }{
        \sqrt{\langle x, y\rangle_M^2-1}
    }
    \left(
        y
        +
        \langle x, y \rangle_M x
    \right)
    ,
\end{align*}
for any $x, y \in \cH^n$ and $v \in \tangentSpace{x}[\cH^n]$; for more details see~\cite[Appendix~C]{BergmannPerschSteidl:2016:1}.
With this we can calculate the geodesic arc from $x$ to $y$ as
\begin{align*}
    \gamma_{x,y}(t)
    &=
    \left(
        \cosh(
            t
            \,
            \text{arcosh}\left(-\langle x, y\rangle_M \right)
        )
        +
        \langle x, y\rangle_M
        \frac{
            \sinh(
                t
                \,
                \text{arcosh}\left(-\langle x, y\rangle_M \right)
            )
        }{
            \sqrt{
                \langle x, y\rangle_M^2-1
            }
        }
    \right)
    x
    \\
    &
    \qquad\qquad\qquad
    +
    \left(
    \frac{
        \sinh(
            t
            \,
            \text{arcosh}\left(-\langle x, y\rangle_M \right)
        )
    }{
        \sqrt{
            \langle x, y\rangle_M^2-1
        }
    }
    \right)
    y
    .
\end{align*}
Using this representation of geodesic arcs and the triangle inequality, one can show that the $\ell_1$-norm is geodesically convex on $(\cH^n, g_M)$, \eg by Taylor expanding the coefficients of $x$ and $y$ in the expression for $\gamma_{x,y}(t)$ around $t = 0$ and $t=1$.

The rest of this section is devoted to show that the proximal operator of the $\ell_1$-norm, defined by
\begin{equation}\label{eq:hyperbolic-l1-prox}
    \prox_{\mu \Vert \cdot \Vert_1}(x)
    \coloneq
    \argmin_{y \in \cH^n}
        \Vert y \Vert_1
        +
        \frac{1}{2\mu} \dist^2(x,y),
\end{equation}
can be computed as $p_x(t^*)$, with $t^*$ a fixed point of $\sigma_x\circ p_x$, where
\begin{align*}
    \sigma_x(y)
    \coloneq
    \mu
    \frac{
        \sqrt{\langle x, y\rangle_M^2-1}
    }{
        \text{arcosh}\left(-\langle x, y\rangle_M \right)
    }
    ,
    &
    \quad
    p_x(t)
    \coloneq
    \frac {
        \bar{p}_x(t)
    }{
        \sqrt{
            - \langle \bar{p}_x(t), \bar{p}_x(t)\rangle_M
        }
    }\in \cH^n,
\end{align*}
and $\bar p_x(t)$ is defined elementwise as
\begin{align*}
    (\bar{p}_x(t))_i
    \coloneq
    \begin{cases}
        x_i + t & \text{if } i = n+1,\\
        \max\{0, x_i - \text{sign}(x_i) t\} & \text{otherwise}.
    \end{cases}
\end{align*}
With this result, we can compute the solution of the prox-subproblem \cref{eq:hyperbolic-l1-prox} using a fixed point iteration that in our experience only takes a few steps to converge to the desired accuracy.

First, we show that such fixed points exist and can be easily computed via a fixed point iteration.
\begin{theorem}
    \label{thm:fixed-point-iteration}
    Let $x\in \cH^n$ and $\mu \ge 0$.
    Then there exists a fixed point $t^*$ of $T \coloneq \sigma_x\circ p_x$ and every sequence $t_k = T(t_{k-1})$ with $t_0 \ge 0$ converges to $t^*$.
\end{theorem}
\begin{proof}
    Define
    $x_{\max} =
        \max_{i = 1, \dots, n} \abs{x_i}
    $.
    First, we observe that for $t \ge x_{\max}$, we have $p_x(t) = (0, \dots, 0, 1)$ and thus $T(t) =  \mu \frac {\sqrt{x_{n+1}^2 - 1}} {\text{arcosh}(x_{n+1})} =: t_{\max} \le t$.
    Note that $\lim_{s \to -1} \frac {\sqrt{s^2 - 1}} {\text{arcosh}(- s)} = 1$, hence $T(0) = \mu \ge 0$.

    We will now show that $T$ is strictly increasing on the interval $[0, x_{\max}]$.
    Together with $T$ being continuous, this implies that $T$ has exactly one fixed point $t^* \in [0, t_{\max}]$.
    We can assume \wolog that $\abs{x_1}\le \dots \le \abs{x_n}$.
    Notice that $p_x(t)$ is smooth on the smaller intervals $I_k = \left(\abs{x_{k-1}}, \abs{x_k}\right)$ for $k = 1,\dots n$ and $x_0 = 0$.
    Consider the function $\phi_k(t) = - \langle p_x(t), x\rangle_M$ on the interval $I_k$, then
    \begin{equation*}
        \phi_k(t) =
        \frac{
             x_{n+1}(x_{n+1}+t) - \sum_{i=k}^{n} x_i^2 - t \abs{x_i}
        }{
            \sqrt{
               (x_{n+1}+t)^2 - \sum_{i=k}^{n} (\abs{x_i} - t)^2
            }
        }
    \end{equation*}
    and
    \begin{align*}
        \phi_k'(t)
        &=
        \frac{
            \left(\sum_{i=k}^{n+1} \abs{x_i}\right)
            \sqrt{
               (x_{n+1}+t)^2 - \sum_{i=k}^{n} (\abs{x_i} - t)^2
            }
            }{
            (x_{n+1}+t)^2 - \sum_{i=k}^{n} (\abs{x_i} - t)^2
        }
        \\
        & \qquad \qquad
        -
        \frac{
                \left(
                    \left(\sum_{i=k}^{n+1} \abs{x_i}\right)
                    -(n-k)t
                \right)
                \left(
                    x_{n+1}(x_{n+1}+t) - \sum_{i=k}^{n} x_i^2 - t \abs{x_i}
                \right)
            }{
            \sqrt{
                    (x_{n+1}+t)^2 - \sum_{i=k}^{n} (\abs{x_i} - t)^2
            }
            \left(
                (x_{n+1}+t)^2 - \sum_{i=k}^{n} (\abs{x_i} - t)^2
            \right)
        }
        \\
        &
        \ge
        \sum_{i=k}^{n+1} \abs{x_i}
        \frac{
            (x_{n+1}+t)^2 - \sum_{i=k}^{n} (\abs{x_i} - t)^2
            -
            x_{n+1}(x_{n+1}+t) + \sum_{i=k}^{n} x_i^2 - t \abs{x_i}
            }{
            \sqrt{
                    (x_{n+1}+t)^2 - \sum_{i=k}^{n} (\abs{x_i} - t)^2
            }
            \left(
                (x_{n+1}+t)^2 - \sum_{i=k}^{n} (\abs{x_i} - t)^2
            \right)
        }
        \\
        &
        >
        0
        .
    \end{align*}
    Thus, $\phi_k$ is strictly increasing on $[0, x_{\max}]$.
    Since $\frac{\sqrt{s^2-1}}{\text{arcosh}(s)}$ is also strictly increasing for $s\ge 1$, it follows that $T$ is strictly increasing on $[0, x_{\max}]$.

    Now, consider the sequence defined by $t_k = T(t_{k-1})$ for some initial value $t_0$.
    If $t_k < t^*$, then ${t_k<T(t_k)<t^*}$, so $\abs{T(t_k)-t^*} < \abs{t_k-t^*}$.
    If $t_k > t^*$ then $T(t_k)<t_k$ and $T(t_k)\ge t^*$, with equality if $t^*=t_{\max}$. Thus, in this case as well, $\abs{T(t_k)-t^*} < \abs{t_k-t^*}$.
    Therefore, $t_k$ converges to $t^*$.
\end{proof}
% Moved table for imanum
\begin{table}[t]
    \centering
    \hspace{-8em} % manually center table for imanum
    \begin{adjustbox}{max width=\textwidth}
    \begin{filecontents*}{data/sparse-approx/results-Hn-time-iter-10-500.csv}
μ,n,CRPG_iter,CRPG_time,CRPG_bt_iter,CRPG_bt_time,CPPA_iter,CPPA_time
0.1,2,151,0.03108788945,36,0.010941200049999999,5000,2.794857123
0.1,10,71,0.018833641750000005,16,0.005922629200000001,5000,3.1050391145500003
0.1,100,36,0.026580250050000004,883,3.9568086188500002,5000,5.7586319918
0.1,500,31,0.12059559180000004,4614,177.5270623792,5000,26.0804579166
0.5,2,114,0.0226985978,28,0.00840331045,5000,2.7997910249
0.5,10,59,0.013369522850000001,11,0.0041409915999999995,4502,2.7899892041500003
0.5,100,33,0.02301437925,616,2.6453486914,5000,5.724037020800001
0.5,500,18,0.0729229042,2588,125.6396339918,1515,7.9726690834500005
1.0,2,47,0.009613120900000003,11,0.0035857998500000003,2013,1.13443696665
1.0,10,32,0.007193241650000001,8,0.0032558419000000005,2511,1.54804101875
1.0,100,14,0.009373489750000002,297,1.3033762165,1018,1.15092910215
1.0,500,8,0.030481054300000004,576,29.300956118749998,22,0.1129241937
\end{filecontents*}
\pgfplotstabletypeset[
        col sep = comma,
        every head row/.style = {
            before row = \toprule
                \multicolumn{2}{c}{} & \multicolumn{4}{c}{CRPG} & \multicolumn{2}{c}{CPPA}\\
                \multicolumn{2}{c}{Settings} & \multicolumn{2}{c}{Constant stepsize} &\multicolumn{2}{c}{Backtracking} & \\
                \cmidrule(lr){1-2}\cmidrule(lr){3-4}\cmidrule(lr){5-6}\cmidrule(lr){7-8},
            after row = \midrule
        },
        % Add lines to separate the experiments with different t
        every row no 3/.style = {after row = \midrule},
        every row no 7/.style = {after row = \midrule},
        % Add bottomrule to the last row
        every last row/.style = {after row = \bottomrule},
        % Column formatting
        display columns/0/.style = {
            column name = $\mu$,
            string type,
            column type = {S[table-number-alignment = right, table-format = 1.1, table-alignment-mode = format]},
            % Only show the first value of t per each group
            postproc cell content/.append code={
                \ifnum\pgfplotstablerow=0
                    \pgfkeyssetvalue{/pgfplots/table/@cell content}{##1}
                \else
                    \ifnum\pgfplotstablerow=4
                        \pgfkeyssetvalue{/pgfplots/table/@cell content}{##1}
                    \else
                        \ifnum\pgfplotstablerow=8
                            \pgfkeyssetvalue{/pgfplots/table/@cell content}{##1}
                        \else
                            \pgfkeyssetvalue{/pgfplots/table/@cell content}{}
                        \fi
                    \fi
                \fi
            }
        },
        display columns/1/.style = {
            column name = $n$,
            string type,
            column type = {S[table-number-alignment = right, table-format = 3, table-alignment-mode = format]},
        },
        display columns/2/.style = {
            column name = Iter.,
            string type,
            column type = {S[table-number-alignment = right, table-auto-round = true, table-format = 5, table-alignment-mode = format]},
        },
        display columns/3/.style = {
            column name = Time (sec.),
            string type,
            column type = {S[table-number-alignment = right, table-auto-round = true, scientific-notation = engineering, exponent-mode = scientific, table-format = 1.2e1]},
        },
        display columns/4/.style = {
            column name = Iter.,
            string type,
            column type = {S[table-number-alignment = right, table-auto-round = true, table-format = 5, table-alignment-mode = format]},
        },
        display columns/5/.style = {
            column name = Time (sec.),
            string type,
            column type = {S[table-number-alignment = right, table-auto-round = true, scientific-notation = engineering, exponent-mode = scientific, table-format = 1.2e1]},
        },
        display columns/6/.style = {
            column name = Iter.,
            string type,
            column type = {S[table-number-alignment = right, table-auto-round = true, table-format = 5, table-alignment-mode = format]},
        },
        display columns/7/.style = {
            column name = Time (sec.),
            string type,
            column type = {S[table-number-alignment = right, table-auto-round = true, scientific-notation = engineering, exponent-mode = scientific, table-format = 1.2e1]},
        },
        multicolumn names,
        % mark by best iterations
        every row 0 column 4/.style={postproc cell content/.style={@cell content={\color{table-highlight-best}##1}}},
        % every row 1 column 2/.style={postproc cell content/.style={@cell content={\color{table-highlight-best}##1}}},
        every row 1 column 4/.style={postproc cell content/.style={@cell content={\color{table-highlight-best}##1}}},
        every row 2 column 2/.style={postproc cell content/.style={@cell content={\color{table-highlight-best}##1}}},
        every row 3 column 2/.style={postproc cell content/.style={@cell content={\color{table-highlight-best}##1}}},
        every row 4 column 4/.style={postproc cell content/.style={@cell content={\color{table-highlight-best}##1}}},
        % every row 5 column 2/.style={postproc cell content/.style={@cell content={\color{table-highlight-best}##1}}},
        every row 5 column 4/.style={postproc cell content/.style={@cell content={\color{table-highlight-best}##1}}},
        every row 6 column 2/.style={postproc cell content/.style={@cell content={\color{table-highlight-best}##1}}},
        every row 7 column 2/.style={postproc cell content/.style={@cell content={\color{table-highlight-best}##1}}},
        every row 8 column 4/.style={postproc cell content/.style={@cell content={\color{table-highlight-best}##1}}},
        every row 9 column 4/.style={postproc cell content/.style={@cell content={\color{table-highlight-best}##1}}},
        every row 10 column 2/.style={postproc cell content/.style={@cell content={\color{table-highlight-best}##1}}},
        every row 11 column 2/.style={postproc cell content/.style={@cell content={\color{table-highlight-best}##1}}},
        % mark by least time
        every row 0 column 5/.style={postproc cell content/.style={@cell content={\color{table-highlight-best}##1}}},
        every row 1 column 5/.style={postproc cell content/.style={@cell content={\color{table-highlight-best}##1}}},
        every row 2 column 3/.style={postproc cell content/.style={@cell content={\color{table-highlight-best}##1}}},
        every row 3 column 3/.style={postproc cell content/.style={@cell content={\color{table-highlight-best}##1}}},
        every row 4 column 5/.style={postproc cell content/.style={@cell content={\color{table-highlight-best}##1}}},
        every row 5 column 5/.style={postproc cell content/.style={@cell content={\color{table-highlight-best}##1}}},
        every row 6 column 3/.style={postproc cell content/.style={@cell content={\color{table-highlight-best}##1}}},
        every row 7 column 3/.style={postproc cell content/.style={@cell content={\color{table-highlight-best}##1}}},
        every row 8 column 5/.style={postproc cell content/.style={@cell content={\color{table-highlight-best}##1}}},
        every row 9 column 5/.style={postproc cell content/.style={@cell content={\color{table-highlight-best}##1}}},
        every row 10 column 3/.style={postproc cell content/.style={@cell content={\color{table-highlight-best}##1}}},
        every row 11 column 3/.style={postproc cell content/.style={@cell content={\color{table-highlight-best}##1}}},
    ]
    {data/sparse-approx/results-Hn-time-iter-10-500.csv}
  \end{adjustbox}
  \vspace{1em}
	\caption{
        Results from \cref{subsection:numerics-sparse-approximation} for sparse approximation on $\cH^n$, averaged over $10$ runs.
    }
	\label{table:CRPG-SPCA-comparisons-time-iterations}
\end{table}

Now that we have established the existence of the fixed point of $T$ and shown how to compute it, we can use this result to find the solution of the proximal operator.
\begin{theorem}
    \label{thm:prox-l1-hyperbolic}
    Let $x\in \cH^n$ and $\mu \ge 0$.
    Then $ \prox_{\mu \Vert \cdot \Vert_1}(x) = p_x(t^*)$, with $t^*$ the unique fixed point of $T = \sigma_x\circ p_x$.
\end{theorem}
\begin{proof}
    As discussed in \Cref{section:preliminaries},
    $y^* = \prox_{\mu \Vert \cdot \Vert_1}(x)$
    is characterized by
    $0_{y^*} \in -\logarithm{y^*}(x) + \mu \partial^{\cH^n} \Vert y^*\Vert_1$ in a unique way.
    First, we show $\proj{y}(J\partial^{\mathbb R^{n+1}} \norm{y}_1) \subseteq \partial^{\cH^n} \norm{y}_1$, where
    \begin{equation*}
        \partial^{\mathbb R^{n+1}} \norm{y}_1
        =
        \setDef[auto]{v \in \mathbb R^{n+1}}{
            v_i \in
            \begin{cases}
            \{\text{sign}(y_i)\} & \text{ if } y_i \ne 0\\
            [-1,1] & \text{ if } y_i = 0
            \end{cases}, \;
            i = 1, \dots, n+1
        }
        ,
    \end{equation*}
    is the Euclidean subgradient of $\norm{y}_1$.
    First, note that for any $v \in \partial^{\mathbb R^{n+1}} \norm{y}_1$, we have
    $\langle v,y\rangle = \norm{y}_1$.
    Thus, $\norm{y}_1 +\langle v, z-y\rangle =  \langle v, z\rangle$.
    Hence, $v \in \partial^{\mathbb R^{n+1}} \norm{y}_1$ if and only if
    $
        \norm{z}_1 \ge \langle v, z\rangle
    $
    for all $z \in \mathbb R^{n+1}$.
    Let $z \in \cH^n$ and write
    $\phi(s) = \frac{
                \text{arcosh}\left(s \right)
            }{
                \sqrt{s^2-1}
            }
    $, then
    \begin{align*}
        \norm{z}_1
        &\ge
        \norm{\phi(-\langle y, z\rangle_M) z}_1
        \\
        &\ge
        \langle Jv, \phi(-\langle y, z\rangle_M) z\rangle_M
        \\
        &=
        \langle
            Jv
            ,
            \phi(-\langle y, z\rangle_M)
            \proj{y}(z)
        \rangle_M
        -
        \langle
            Jv
            ,
            \phi(-\langle y, z\rangle_M)
            \langle y, z\rangle_M y
        \rangle_M
        \\
        &=
        \langle
            \proj{y}(Jv)
            ,
            \logarithm{y}{z}
        \rangle_M
        +
        (-\langle y, z\rangle_M)\phi(-\langle y, z\rangle_M)
        \langle
            v
            ,
            y
        \rangle
        \\
        &\ge
        \langle
            \proj{y}(Jv)
            ,
            \logarithm{y}{z}
        \rangle_M
        +
        \norm{y}_1,
    \end{align*}
    where we used that for $s\ge 1$, $\phi(s) \le 1$ and $s\phi(s)\ge 1$.
    Thus, $\proj{y}(Jv) \in \partial^{\cH^n} \norm{y}_1$.

    Let now $y^* =  p_x(t^*)$, where $t^*$ the fixed point of $\sigma_x\circ p_x$ and $v^* \in \partial^{\mathbb R^{n+1}} \norm{y^*}_1$ defined elementwise by $v^*_i = \text{sign}(y^*_i)$ if $y^*_i \ne 0$ and $v^*_i = \frac {x_i} {t^*}$ if $y^*_i = 0$.
    By rewriting the logarithmic map as a scaled projection onto the tangent space and using that $\sigma(y^*) = t^*$, we get that $\logarithm{y^*}(x) = \mu \proj{y}(Jv^*)$ is equivalent to
    \begin{equation*}
        x-t^*Jv^* = - \langle y^*, x-t^*Jv^*\rangle_M y^*.
    \end{equation*}
    This equality holds because, by construction, $ x-t^*Jv^* = y^*$ and $\langle y^*, y^*\rangle_M = -1$.
    Therefore, ${\logarithm{y^*}(x) \in \mu \proj{y}(J\partial^{\mathbb R^{n+1}} \norm{y^*}_1) \subseteq \mu \partial^{\cH^n} \Vert y^*\Vert_1}$.
\end{proof}

\paragraph{\textbf{Implementation Details}}
\label{subsection:numerics-implementation-PCA}
\footnote{The code is available at \url{https://juliamanifolds.github.io/ManoptExamples.jl/stable/examples/CRPG-Sparse-Approximation/}}
We tested the CRPG method using a dataset that was generated as follows.
We picked a random anchor point $\bar q \in \cM$ and then generated $N=1000$ Gaussian tangent vectors at $\bar q$ with standard deviation $\sigma = 1$.
Then, we obtained $N$ points $q_i \in \cM$ by applying the exponential map to the sampled tangent vectors.

We ran our \cref{algorithm:CRPG} using the fixed point iteration from \cref{thm:fixed-point-iteration} to solve the proximal subproblem.
We estimated the diameter of a geodesic ball around the initial point that contains all points $q_i$ and $\bar q$, and we used this number to set the constant stepsize $\lambda$.
This is because if $\cU$ is a geodesically convex set containing all data points $q_j$, the function $g$ is $\zetafirst{D}$-smooth over $\cU$, where $D = \diam(\cU)$; see, \eg,~\cite[Lemma~5]{ZhangSra:2016:1}.
We thus use $\lipgrad = \zeta_{1, -1}(D)$ for the Lipschitz constant of $\grad g$, and hence the constant step size is $\lambda = \frac{1}{\lipgrad}$ in \cref{algorithm:CRPG}.
For the backtracking procedure, we set the initial guess to $s = 4\lambda$, and we used a contraction factor of $\eta = 0.8$ and a warm start factor of $\theta = 10$.

As a comparison, we tested the Cyclic Proximal Point Algorithm (CPPA)~\cite{Bacak:2014:1} with the same dataset.
The stopping criterion for the CPPA was set to a check on the change in the objective function value between two consecutive iterations, with a tolerance of $10^{-7}$ and a cap of at most $5000$ iterations.

Each run was initialized with a random starting point using the uniform distribution on $\cM$.
The fixed point iteration for the proximal subproblem was terminated when either $\vert \sequence{t}{k+1} - \sequence{t}{k} \vert < 10^{-7}$ or after $20$ iterations.
It usually terminated after only a handful of iterations.
\paragraph{\textbf{Discussion}}
The results of the numerical experiment are shown in \cref{table:CRPG-SPCA-comparisons-time-iterations}.
The first column displays the sparsity parameter $\mu$.
As can be read from the table, the runs with a constant stepsize converge in less iterations and are faster time-wise than both the backtracked ones and the CPPA on the experiments in dimensions $100$ and $500$. The backtracking strategy performs the best in the low-dimensional experiments.
The backtracking procedure performs poorly for larger dimensions in this scenario.
There is no meaningful difference between the objective values and sparsities obtained by the three algorithms.

\subsection{Constrained Mean on $\cH^n$}
\label{subsection:constrained-mean-example}

The following example is meant as a continuation of~\cite[Section~6]{BergmannFerreiraNemethZhu:2025:1} which compares their Projected Gradient Algorithm (PGA) with the Augmented Lagrangian Method~\cite[Algorithm~1]{LiuBoumal:2019:2} and the Exact Penalty Method~\cite[Algorithm~2]{LiuBoumal:2019:2}.
We expand on their experiment and compare the PGA to CRPG.

We consider the problem of finding the Riemannian center of mass of a set of $N=400$ distinct points $\{p_1, \ldots, p_N\} \subset \cU$, with $\cU \subset \cH^n$ a geodesically convex set, constrained to a geodesically convex ball $\cB(p_0, r) \subset \cU$ with center $p_0$ and radius $r$.
This can be phrased as an unconstrained optimization problem of the form \cref{eq:splitting} with
\begin{equation*}
g(p) = \frac{1}{2 N} \sum_{i=1}^{N} \dist^2(p, p_i)
\quad
\text{and}
\quad
h(p) = \chi_{\cB(p_0, r)}(p)
,
\end{equation*}
where $\chi_{\cB(p_0, r)}$ is the characteristic function of the geodesically convex ball $\cB(p_0, r)$, that takes the value of zero if $p \in \cB(p_0, r)$ and infinity otherwise.
The function $g$ is strongly geodesically convex, which makes $f = g+h$ strongly geodesically convex too, with gradient $\grad g(p) = -\frac{1}{N} \sum_{i=1}^{N} \logarithm{p}{p_i}$.
The proximal mapping of $h$ is given by the projection onto the ball $\cB(p_0, r)$, that is
\begin{equation}
    \label{eq:proximal-mapping-ball}
    \prox_{\lambda h}(p)
    =
    \begin{cases}
        p
        &
        \text{if } p \in \cB(p_0, r)
        ,
        \\
        \geodesic<l>[auto]{p_0}{p}(
            \frac{r}{\dist(p, p_0)}
        )
        &
        \text{otherwise}
        .
    \end{cases}
\end{equation}

\begin{figure}[tbp]
    \pgfplotstableread[col sep = comma]{data/constrained-mean/CRPG-CnBallConstrMean-2-200-times.csv}\datatabletimes
    \pgfplotstableread[col sep = comma]{data/constrained-mean/CRPG-CnBallConstrMean-2-200-iterations.csv}\datatableiterations
    \begin{subfigure}{.5\textwidth}
        \centering
        \begin{tikzpicture}
            \begin{semilogyaxis}[
                width=0.99\textwidth,
                xlabel={dimension},
                ylabel={time (s)},
                ylabel near ticks,
                xlabel near ticks,
                ymajorgrids,
                label style={font=\scriptsize},
                tick label style={font=\scriptsize},
                legend style={font=\scriptsize, nodes={scale=0.66}},
                legend pos=north east,
                legend cell align=left,
            ]
                \addplot[color=TolBrightBlue,thick] table[x = d, y = crpg_cn] {\datatabletimes};
                \addplot[color=TolBrightCyan,thick] table[x = d, y = crpg_bt] {\datatabletimes};
                \addplot[color=TolBrightRed,thick] table[x = d, y = pga] {\datatabletimes};
            \end{semilogyaxis}
        \end{tikzpicture}
    \caption{CPU runtime for convergence as a function of the manifold dimension.} %break line for amspreprint
    \label{subfig:constr-mean-times}
    \end{subfigure}%
    \begin{subfigure}{.5\textwidth}
        \centering
        \begin{tikzpicture}
            \begin{axis}[
                width=0.99\textwidth,
                xlabel={dimension},
                ylabel={iterations},
                ylabel near ticks,
                xlabel near ticks,
                ymajorgrids,
                label style={font=\scriptsize},
                tick label style={font=\scriptsize},
                legend style={font=\scriptsize, nodes={scale=0.66}},
                legend pos=north east,
                legend cell align=left,
            ]
                \addplot[color=TolBrightBlue,thick] table[x = d, y = crpg_cn] {\datatableiterations};
                \addplot[color=TolBrightCyan,thick] table[x = d, y = crpg_bt] {\datatableiterations};
                \addplot[color=TolBrightRed,thick] table[x = d, y = pga] {\datatableiterations};
                \legend{CRPG constant stepsize, CRPG backtracking, PGA}
            \end{axis}
        \end{tikzpicture}
        \caption{Iterations to convergence as a function of the manifold dimension.}
        \label{subfig:constr-mean-iterations}
    \end{subfigure}%
    \caption{Comparison of CRPG with PGA for the constrained mean problem on $\cH^n$ from \cref{subsection:constrained-mean-example}, averaged over $10$ runs.}
\label{fig:constr-mean}
\end{figure}

\paragraph{\textbf{Implementation Details}}
\footnote{The code is available at \url{https://juliamanifolds.github.io/ManoptExamples.jl/stable/examples/CRPG-Constrained-Mean-Hn/}}
As in \cref{subsection:numerics-sparse-approximation}, since $\cU$ is a geodesically convex set containing all data points $p_j$, $g$ is $\zeta_{1, \kmin}(D)$-smooth over $\cU$, where $D = \diam(\cU)$.
We thus use $\lipgrad = \zeta_{1, -1}(D)$ for the Lipschitz constant of $\grad g$, and hence the constant step size is $\lambda = \frac{1}{\lipgrad}$ in \cref{algorithm:CRPG}.
For the backtracking procedure, we use an initial guess $s = 4 \lambda$, while the contraction factor is set to $\eta = 0.8$ and the warm start factor to $\theta = 10$.
We average each experiment over $10$ runs where the original starting point is perturbed with a random vector tangent to it, sampled from a normal distribution with standard deviation $\sigma = 1$.

\paragraph{\textbf{Discussion}}
The results of the numerical experiment are shown in \cref{fig:constr-mean}.
The left plot shows the CPU runtime in seconds as a function of the manifold dimension, while the right plot shows the number of iterations to convergence.
As can be seen from the plots, CRPG with a constant stepsize converges faster in time than the PGA.
The PGA in turn converges with a rate slightly slower than that of CRPG with backtracking.
The number of iterations is similar for CRPG in both variants, with the CRPG with constant stepsize requiring slightly fewer iterations than the backtracked case.
Interestingly, CRPG requires more iterations than the PGA, especially when the dimension is low, but still converges faster in terms of CPU runtime.

\section{Summary \& Conclusion}
\label{section:conclusion}

In this article, we introduced a Riemannian proximal gradient method that works intrinsically on a Riemannian manifold for solving (possibly strongly) geodesically convex problems.
We established state-of-the-art convergence results on manifolds with bounded curvature with standard geodesic convexity assumptions, thus generalizing the Hadamard setting.
We have demonstrated the efficiency of our method with several numerical experiments.
We also generalized the proximal gradient inequalities to Riemannian manifolds, which allowed us to derive known convergence rates for the proximal point algorithm and the gradient descent method on Hadamard manifolds.

Further research directions include investigating theoretical aspects of retraction-based and accelerated variants of CRPG, and studying other splitting methods based on proximal maps for solving optimization problems intrinsically on Riemannian manifolds.

{\subsection*{Acknowledgements}
\small
P.J.~was funded by the DFG – Projektnummer 448293816.}

% Insert the appendix.
\appendix
\section{Proofs of \cref{section:proximal-gradient-inequalities}}
\label{appendix}

\subsection{Proof of \cref{thm:prox-grad-inequality}}
\label{appendix:proof-prox-grad-inequality}

    Let $q \in \interior(\dom(g))$.
    Then $\Imap[auto]{\lambda}{q} = \prox_{\lambda h}(\gradientstep{q})$, and by~\cite[Lemma~4.2]{FerreiraOliveira:2002:1}, we obtain 
    ${
        \frac{1}{\lambda}
        \logarithm{\Imap[auto]{\lambda}{q}}{
            \gradientstep{q}
        }
        \in
        \partial h(\Imap[auto]{\lambda}{q})
        ,
    }$
    which implies
    \begin{equation*}
        h(p)
        \ge
        h(\Imap[auto]{\lambda}{q})
        +
        \frac{1}{\lambda}
        \riemannian{
            \logarithm{\Imap[auto]{\lambda}{q}}{
            \gradientstep{q}
            }
        }{
            \logarithm{\Imap[auto]{\lambda}{q}}{p}
        }
        ,
    \end{equation*}
    for all $p \in \cM$, since $h$ is geodesically convex.
    By adding $g(p)$ and subtracting $g(\Imap[auto]{\lambda}{q})$ from both sides and rearranging the terms, we get
    $
        f(p)
        -
        f(\Imap[auto]{\lambda}{q})
        \ge
        g(p)
        -
        g(\Imap[auto]{\lambda}{q})
        +
        \frac{1}{\lambda}
        \riemannian{
            \logarithm{\Imap[auto]{\lambda}{q}}\gradientstep{q}
        }{
            \logarithm{\Imap[auto]{\lambda}{q}}p
        }
        ,
    $
    and by \cref{eq:prox-grad-assumption} we obtain
    \begin{equation}
        \label{eq:first-prox-grad-bound}
        \begin{aligned}
            f(p)
            -
            f(\Imap[auto]{\lambda}{q})
            &
            \ge
            \glinear(p, q)
            +
            \frac{1}{\lambda}
            \riemannian{
                \logarithm{q}{\gradientstep{q}}
            }
            {
                \logarithm{q}{\Imap[auto]{\lambda}{q}}
                -
                \logarithm{q}{p}
            }
            \\
            &
            \quad
            -
            \frac{1}{2 \lambda}
            \dist^2(
                q
                ,
                \Imap[auto]{\lambda}{q}
            )
            +
            \frac{1}{\lambda}
            \riemannian{
                \logarithm{\Imap[auto]{\lambda}{q}}\gradientstep{q}
            }{
                \logarithm{\Imap[auto]{\lambda}{q}}p
            }
            ,
        \end{aligned}
    \end{equation}
    where the last inequality follows from the definition of $\glinear$ and $\gradientstep{q}$, and from rearranging the terms.
    We now apply the cosine laws, see, \eg,~\cite[Corollary~15]{MartinezRubioPokutta:2023:1}, to the inner product terms as follows:
    \begin{align*}
        2
        \riemannian{
            \logarithm{q}{\gradientstep{q}}
        }
        {
            \logarithm{q}{\Imap[auto]{\lambda}{q}}
        }
        &
        \ge
        \zeta_{2, \kmax}(D_1)
        \dist^2(q, \gradientstep{q})
        +
        \dist^2(q, \Imap[auto]{\lambda}{q})
        -
        \dist^2(\gradientstep{q}, \Imap[auto]{\lambda}{q})
        ,
    \end{align*}
    where $D_1$ is the diameter of the uniquely geodesic triangle with vertices $q$, $\gradientstep{q}$, and $\Imap[auto]{\lambda}{q}$.
    Similarly, we have
    \begin{equation}
        \label{eq:cosine-law-bound-2}
        \begin{aligned}
            -
            2
            \riemannian{
                \logarithm{q}\gradientstep{q}
            }
            {
                \logarithm{q}p
            }
            &
            \ge
            -
            \dist^2(\gradientstep{q}, q)
            -
            \zeta_{1, \kmin}(D_2)
            \dist^2(p, q)
            +
            \dist^2(p, \gradientstep{q})
            ,
        \end{aligned}
    \end{equation}
    with $D_2$ the diameter of the uniquely geodesic triangle with vertices $q$, $\gradientstep{q}$, and $p$.
    Finally, we get
    \begin{align*}
        2
        \riemannian{
            \logarithm{\Imap[auto]{\lambda}{q}}\gradientstep{q}
        }
        {
            \logarithm{\Imap[auto]{\lambda}{q}}p
        }
        &
        \ge
        \zeta_{2, \kmax}(D_3)
        \dist^2(\gradientstep{q}, \Imap[auto]{\lambda}{q})
        +
        \dist^2(p, \Imap[auto]{\lambda}{q})
        -
        \dist^2(p, \gradientstep{q})
        ,
    \end{align*}
    where $D_3$ is the diameter of the uniquely geodesic triangle with vertices $\Imap[auto]{\lambda}{q}$, $\gradientstep{q}$, and $p$.
    Plugging these bounds back into \cref{eq:first-prox-grad-bound} yields \cref{eq:prox-grad-inequality-1}

    Analogously, using the cosine inequalities with different vertices, we can bound the inner product terms in \cref{eq:first-prox-grad-bound} to obtain \cref{eq:prox-grad-inequality-2}.

\subsection{Proof of \cref{cor:fundamental-prox-grad-inequality}}
\label{appendix:proof-fundamental-prox-grad-inequality}

Observe that, on Hadamard manifolds, one has $\kmax \le 0$, which implies $\zeta_{2, \kmax}(r) = 1$ for all $r \in \bbR$.
    Furthermore, the remaining cosine inequality in \cref{eq:cosine-law-bound-2} holds with the tighter constant $\zeta_{1, \kmin}(\dist(q, \gradientstep{q}))$,  see, \eg,~\cite[Remark~16]{MartinezRubioPokutta:2023:1}.
    Combining these facts with \cref{eq:prox-grad-inequality-1} gives
    \begin{equation}
        \label{eq:hadamard-prox-grad-inequality-2}
        f(p)
        -
        f(\Imap[auto]{\lambda}{q})
        \ge
        \glinear(p, q)
        +
        \frac{1}{2 \lambda}
        \dist^2(p, \Imap[auto]{\lambda}{q})
        -
        \frac{\zeta_{1, \kmin}(\dist(q, \gradientstep{q}))}{2 \lambda}
        \dist^2(p, q)
        .
    \end{equation}
    Analogously, \cref{eq:prox-grad-inequality-2} holds with the tighter constant $\zeta_{1, \kmin}(\dist(p, q))$, giving
    \begin{equation}
        \label{eq:hadamard-prox-grad-inequality-3}
        \begin{aligned}
            f(p)
            -
            f(\Imap[auto]{\lambda}{q})
            &
            \ge
            \glinear(p, q)
            +
            \frac{1}{2 \lambda}
            \dist^2(p, \Imap[auto]{\lambda}{q})
            -
            \frac{1}{2 \lambda}
            \dist^2(p, q)
            \\
            &
            \quad
            +
            \frac{1 - \zeta_{1, \kmin}(\dist(p, q))}{2 \lambda}
            \dist^2(q, \gradientstep{q})
            .
        \end{aligned}
    \end{equation}
    Finally, adding \cref{eq:hadamard-prox-grad-inequality-2} and \cref{eq:hadamard-prox-grad-inequality-3} yields the claim.

\subsection{Proof of \cref{thm:rpp-rgd-convergence-rates}}
\label{appendix:proof-convergence-rates}
Because \cref{eq:convex-sufficient-decrease-condition} is satisfied by either stepsize strategy and with $\zetadelta \equiv 1$, substituting $\lambda = \sequence{\lambda}{n}$, $p = p^\ast$, and $q = \sequence{p}{n}$ in the fundamental prox-grad inequality \cref{eq:hadamard-prox-grad-inequality-1} yields
\begin{equation}
    \label{eq:prox-grad-convergence-failure}
    \begin{aligned}
        -
        2 \sequence{\lambda}{n}
        \sequence{\Delta}{n+1}
        &
        \ge
        2 \sequence{\lambda}{n}
        \glinear(p^\ast, \sequence{p}{n})
        +
        \dist^2(
            p^\ast
            ,
            \sequence{p}{n+1}
        )
        \\
        &
        \quad
        -
        \frac{\zeta_{1, \kmin}(\dist(\sequence{p}{n}, \gradientstep{\sequence{p}{n}})) + 1}{2}
        \dist^2(p^\ast, \sequence{p}{n})
        \\
        &
        \quad
        -
        \frac{\zeta_{1, \kmin}(\dist(p^\ast, \sequence{p}{n})) - 1}{2}
        \dist^2(\sequence{p}{n}, \gradientstep{\sequence{p}{n}})
        ,
    \end{aligned}
\end{equation}
where the convexity of $g$ implies $\glinear(p^\ast, \sequence{p}{n}) \ge 0$, which can hence be dropped.

Observe that, if $g = 0$, \ie in the purely RPP case, then $\gradientstep{\sequence{p}{n}} = \sequence{p}{n}$, and \cref{eq:prox-grad-convergence-failure} simplifies to
$
    -
    2 \sequence{\lambda}{n}
    \sequence{\Delta}{n+1}
    \ge
    \dist^2(
        p^\ast
        ,
        \sequence{p}{n+1}
    )
    -
    \dist^2(
        p^\ast
        ,
        \sequence{p}{n}
    )
    .
$
Also note that, if $g=0$, any $C > 0$ satisfies $0 \le \frac{C}{2} \dist^2(p, q)$, so $\lipgrad = C$.
Summing over $n = 0, \ldots, k-1$ and telescoping gives
\begin{equation*}
    -2 \sum_{n=0}^{k-1} \sequence{\lambda}{n} \sequence{\Delta}{n+1}
    \ge
    \dist^2(p^\ast, \sequence{p}{k})
    -
    \dist^2(p^\ast, \sequence{p}{0})
    \ge
    -\dist^2(p^\ast, \sequence{p}{0}).
\end{equation*}
Since the sequence $\sequence{\Delta}{k}$ is nonincreasing,
$
    \sum_{n=1}^{k}
    \sequence{\Delta}{n}
    =
    \sum_{n=0}^{k-1}
    \sequence{\Delta}{n+1}
    \ge
    k
    \sequence{\Delta}{k}
    ,
$
which used in the previous inequality with the lower bound $\sequence{\lambda}{n} \ge \frac{\beta}{\lipgrad}$ from \cref{lemma:convex-stepsize-bounds} yields
\begin{equation*}
    2k \frac{\beta}{\lipgrad}\sequence{\Delta}{k}
    \le
    \dist^2(p^\ast, \sequence{p}{0}),
\end{equation*}
after rearranging the terms.
Setting $R \coloneq \dist(p^\ast, \sequence{p}{0})$, this gives the $\cO\left(\frac{\lipgrad R^2}{2 \betastepsize k}\right)$ convergence rate of the proximal point method.
Compare, \eg,~\cite[Theorem~3.5]{BentoFerreiraMelo:2017:1}.

Similarly, if $h = 0$, \ie in the purely RGD case, then $\gradientstep{\sequence{p}{n}} = \sequence{p}{n+1}$ for all $n$, and \cref{eq:hadamard-prox-grad-inequality-3} becomes
\begin{equation}
    \label{eq:rgd-delta-bound}
    \begin{aligned}
        -
        2 \sequence{\lambda}{n}
        \sequence{\Delta}{n+1}
        &
        \ge
        \dist^2(
            p^\ast
            ,
            \sequence{p}{n+1}
        )
        -
        \dist^2(
            p^\ast
            ,
            \sequence{p}{n}
        )
        -
        \zeta_{1, \kmin}(\dist(p^\ast, \sequence{p}{n}))
        \dist^2(\sequence{p}{n}, \sequence{p}{n+1})
        ,
    \end{aligned}
\end{equation}
dropping the positive terms.
The sufficient decrease condition \cref{eq:sufficient-decrease-second-version-2} with $p = \sequence{p}{n}$, can be multiplied by $\zeta_{1, \kmin}(\dist(p^\ast, \sequence{p}{n}))$, reading
\begin{equation}
    \label{eq:rgd-curvature-sufficient-decrease}
    \zeta_{1, \kmin}(\dist(p^\ast, \sequence{p}{n}))
    \left(
        \sequence{\Delta}{n}
        -
        \sequence{\Delta}{n+1}
    \right)
    \ge
    \frac{\zeta_{1, \kmin}(\dist(p^\ast, \sequence{p}{n}))}{2 \sequence{\lambda}{n}}
    \dist^2(\sequence{p}{n}, \sequence{p}{n+1})
    ,
\end{equation}
and added to \cref{eq:rgd-delta-bound} to cancel out the term involving $\dist^2(\sequence{p}{n}, \sequence{p}{n+1})$ as
\begin{align*}
    2 \sequence{\lambda}{n}
    \left[
        \zeta_{1, \kmin}(\dist(p^\ast, \sequence{p}{n}))
        (
            \sequence{\Delta}{n}
            -
            \sequence{\Delta}{n+1}
        )
        -
        \sequence{\Delta}{n+1}
    \right]
    &
    \ge
    \dist^2(
        p^\ast
        ,
        \sequence{p}{n+1}
    )
    -
    \dist^2(p^\ast, \sequence{p}{n})
    .
\end{align*}
If we now assume a constant stepsize $\lambda$, rearranging the terms and summing over $n = 0, \ldots, k-1$ gives
$
    -
    \zeta_{1, \kmin}(R)
    \sequence{\Delta}{0}
    +
    \zeta_{1, \kmin}(R)
    \sequence{\Delta}{k}
    +
    \sum_{n=1}^{k}
    \sequence{\Delta}{n}
    \le
    \frac{R^2}{2 \lambda}
    ,
$
where we used the fact that the function $\zeta_{1, \kmin}(r)$ is monotonically increasing for $r \ge 0$.
Since the sequence $\sequence{\Delta}{k}$ is nonincreasing, we have again
${
    \sum_{n=1}^{k}
    \sequence{\Delta}{n}
    =
    \sum_{n=0}^{k-1}
    \sequence{\Delta}{n+1}
    \ge
    k
    \sequence{\Delta}{k}
    ,
}$
which combined with the previous inequality yields the second claim.
Compare, \eg,~\cite[Theorem~13]{ZhangSra:2016:1}, for the convergence rate of RGD on Hadamard manifolds.

% Insert the bibliography.
\printbibliography

\end{document}